\documentclass[11pt]{article}
\usepackage{amssymb,latexsym,amsmath,graphicx,amsfonts,amsthm,xy,eucal,mathrsfs,fullpage, epsfig}
\usepackage{amscd,subfigure,color,setspace,cite,mathtools}
\usepackage[font=footnotesize]{caption}
\usepackage{float}
\usepackage[font=footnotesize]{caption}
\usepackage{epstopdf}
\usepackage[english]{babel}
\usepackage{babel}
\usepackage{fontenc}
\usepackage[colorlinks=true,citecolor=blue]{hyperref}
\usepackage{float}
\usepackage{lineno}
\usepackage{pgf,tikz}
\usepackage{array}
\usepackage{multirow}
\usepackage{float}
\restylefloat{table}

\usetikzlibrary{arrows}
\usetikzlibrary[patterns]
\usetikzlibrary{decorations.pathmorphing}
 \usetikzlibrary{plotmarks}
        \usetikzlibrary{intersections}
        \usetikzlibrary{shadings}
        \usetikzlibrary{calc}
\newlength{\hatchspread}
\newlength{\hatchthickness}
\newlength{\hatchshift}
\newcommand{\hatchcolor}{}
\tikzset{hatchspread/.code={\setlength{\hatchspread}{#1}},
         hatchthickness/.code={\setlength{\hatchthickness}{#1}},
         hatchshift/.code={\setlength{\hatchshift}{#1}},
         hatchcolor/.code={\renewcommand{\hatchcolor}{#1}}}
\tikzset{hatchspread=3pt,
         hatchthickness=0.4pt,
         hatchshift=1pt,
         hatchcolor=black}
\pgfdeclarepatternformonly[\hatchspread,\hatchthickness,\hatchshift,\hatchcolor]
   {custom north west lines}
   {\pgfqpoint{\dimexpr-2\hatchthickness}{\dimexpr-2\hatchthickness}}
   {\pgfqpoint{\dimexpr\hatchspread+2\hatchthickness}{\dimexpr\hatchspread+2\hatchthickness}}
   {\pgfqpoint{\dimexpr\hatchspread}{\dimexpr\hatchspread}}
   {
    \pgfsetlinewidth{\hatchthickness}
    \pgfpathmoveto{\pgfqpoint{0pt}{\dimexpr\hatchspread+\hatchshift}}
    \pgfpathlineto{\pgfqpoint{\dimexpr\hatchspread+0.15pt+\hatchshift}{-0.15pt}}
    \ifdim \hatchshift > 0pt
      \pgfpathmoveto{\pgfqpoint{0pt}{\hatchshift}}
      \pgfpathlineto{\pgfqpoint{\dimexpr0.15pt+\hatchshift}{-0.15pt}}
    \fi
    \pgfsetstrokecolor{\hatchcolor}
    \pgfusepath{stroke}
   }
\usetikzlibrary{arrows}
\usetikzlibrary[patterns]
\numberwithin{equation}{section}
\newtheorem{theorem}{Theorem}[section]
\newtheorem{definition}[theorem]{Definition}
\newtheorem{lemma}[theorem]{Lemma}
\newtheorem{example}[theorem]{Example}
\newtheorem{proposition}[theorem]{Proposition}
\newtheorem{corollary}[theorem]{Corollary}
\newtheorem{remark}{Remark}[section]

\def\bbr{{\mathbb R}}

\def\bx{{\bar{x}}}
\def\by{\bar{y}}
\begin{document}
\title{\textbf{Global stability in the Ricker model with delay and stocking}}
\author{Ziyad AlSharawi\thanks{Corresponding author: zsharawi@aus.edu} $\;$ and $\;$ Sadok Kallel \\[2pt]
Department of Mathematics and Statistics\\
 American University of Sharjah, P. O. Box 26666\\
  University City, Sharjah, UAE\\ }
\date{\today}
\maketitle

\begin{abstract}
We consider the Ricker model with delay and constant or periodic stocking.  We found that the high stocking density tends to neutralize the delay effect on stability. Conditions are established on the parameters to ensure the global stability of the equilibrium solution in the case of constant stocking, as well as the global stability of the $2$-periodic solution in the case of $2$-periodic stocking.  Our approach extensively relies on the utilization of the embedding technique. Whether constant stocking or periodic stocking, the model has the potential to undergo a Neimark-Sacker bifurcation in both cases. However, the Neimark-Sacker bifurcation in the $2$-periodic case results in the emergence of two invariant curves that collectively function as a single attractor. Finally, we pose open questions in the form of conjectures about global stability for certain choices of the parameters.
\end{abstract}
\section{ Introduction}
As an early attempt to model a population of a single species with non-overlapping generations, Moran (1950) \cite{Mo1950} proposed using a one-hump map that increases to a maximum and then decreases asymptotically to a non-negative value. Along this line of thought and based on experimental data, Ricker introduced a density-dependent model that became known in the scientific literature as the Ricker model \cite{Ri1954}
\begin{equation}\label{Eq-Ricker1}
x_{n+1}=x_nf(x_n)=x_ne^{r\left(1-\frac{x_n}{K}\right)}, \quad K, r>0 \ \hbox{and}\ x_0\in \bbr_+=[0,+\infty).
\end{equation}
In this model, the parameters $K$ and $r$ represent the carrying capacity and the intrinsic growth rate, respectively. To learn more about the origin and significance of Ricker's model, and Ricker's pivotal contribution in the field of quantitative fishery science, we refer the reader to \cite{Sc2006}.  Rescaling can reduce the model to $y_{n+1}=y_n\exp(r-y_n).$  This model enjoyed a boost of success after the publication of the two studies by May \cite{Ma1974, Ma1976}, and a third study by May and Oster \cite{Ma-Os1976}, in which the model's intriguing dynamics were investigated, and the global stability of its positive equilibrium for $0<r<2$ was observed. A rigorous proof of the global stability for $0<r\leq 2$ was developed later using various techniques such as Lyapunov functions \cite{Fi-Go-Vi1979}, Singer's theorem \cite{Si1978} and enveloping \cite{Cu2005}.
\\

While implicit time lags are embedded in discrete systems, explicit time lags caused by significant recruitment delays must be accounted for in the density-dependent function $f.$ Time delays are necessary to accommodate the recruitment to the adult stage, which may vary based on the type of species  \cite{Pe1974,Le-Ma1976,Br-Ch2012}. This mechanism requires that a mathematical model be considered with a certain time lag $N.$ As a result, Eq. \eqref{Eq-Ricker1} becomes
\begin{equation}\label{Eq-Ricker2}
y_{n+1}=y_nf(y_{n-N})=y_ne^{r-y_{n-N}}.
\end{equation}
Due to implicit delays, solutions of discrete models are observed to overshoot and undershoot their equilibrium level and generate oscillations, periodic solutions, or chaos. This tendency is typically aggravated when explicit delays are incorporated into the system \cite{Pe1974,Le-Ma1976,Ko2001}. When a population exhibits convergence to its equilibrium regardless of its initial density, the phenomenon is called global stability. Proving mathematically that the equilibrium solution is globally stable becomes challenging when time delays are involved. Under general settings of the density-dependent map $f,$ Liz \emph{et al.} \cite{Li-Tk-Vi-Tr2006} found sufficient conditions to obtain global stability for Eq. \eqref{Eq-Ricker2} when $0<r<\frac{3}{2(N+1)}.$ This implies $0<r<\frac{3}{4}$ when $N=1.$  In a lengthy and technical paper that included computer-assisted proofs, Bartha \emph{et al.} \cite{Ba-Ga-Kr2013} established global stability when $N=1$ and $0<r<1.$
\\

Stocking or harvesting can be used as a management tool to achieve various goals, including chaos reversal and re-stabilization of an equilibrium level \cite{Se1998,Se-Ro1998}. According to \cite{Al-Am2016}, constant rate harvesting is observed to transform contest competition into scramble competition. On the other hand, constant stocking slows the fluctuation and has a stabilizing effect on equilibrium solutions \cite{Al2013}.  When constant stocking is applied to the Ricker model with delay,  Eq. (\ref{Eq-Ricker2}) becomes
\begin{equation}\label{Eq-Ricker3}
y_{n+1}=y_nf(y_{n-1})+h=y_ne^{r-y_{n-1}}+h=F(y_n,y_{n-1}),\quad h>0.
\end{equation}
The parameter $h$ can also be regarded as a perturbation parameter, as investigated in  \cite{Br-Ki2006}, but without considering any time delay. Our primary motivation for considering Eq. (\ref{Eq-Ricker3}) stems from two significant factors, namely the effect of delay and stocking on stability. The individual impact of delay and stocking was considered in research \cite{Pe1974,Le-Ma1976,Al2013,Ko2001}, and it is known to reflect contrasting impacts on stability. However, exploring the interaction between delay and stocking and their impact on overall stability is a novel task. 
\\

Equations of the form $y_{n+1}=y_nf(y_{n-1})+h$ have been considered in \cite{Al-Am2016} under the assumption that $tf(t)$ is increasing; however, Eq. (\ref{Eq-Ricker3}) does not belong to this category.  If the stocking is done through seasonal quotas, the constant $h$ is replaced by a sequence $\{h_n\}$, which we assume to be $p$-periodic.
 Therefore, Eq. (\ref{Eq-Ricker3}) becomes
\begin{equation}\label{Eq-Ricker4}
y_{n+1}=y_nf(y_{n-1})+h_n=y_ne^{r-y_{n-1}}+h_n=F_n(y_n,y_{n-1}),
\end{equation}
where $h_n\geq 0$ is a $p$-periodic sequence and $F$ is defined on the positive orthant $\bbr^2_+$. We continue to write $f(y)$ instead of $e^{r-y}$ whenever we find it convenient. Unlike Eq. \eqref{Eq-Ricker1}, which has two equilibrium points $x=0$ and $x=r$, Eq. (\ref{Eq-Ricker3}) has a unique non-negative equilibrium solution, which bifurcates into a $p$-periodic solution when periodicity is introduced in Eq. (\ref{Eq-Ricker4}).
\\

The structure of our paper is focused on the development of the theory and apparatus required to investigate the global stability of the equilibrium solution in Eq. (\ref{Eq-Ricker3}) and the $p$-periodic solution in   Eq. (\ref{Eq-Ricker4}). We provide a clear geometric understanding of the reason why stability (local or global) holds. The sections of this paper are structured as follows: In Section Two, preliminary findings regarding the utilization of the embedding technique in autonomous and periodic cases are presented.
 In the third section, we leverage the outcomes from Section Two to apply them to Equation (\ref{Eq-Ricker3}). This section delineates sufficient conditions on the variable $h$ that ensure global stability in the autonomous case. We determine ``almost completely" the local and global stability regions in the $(h,r)$ plane, modulo a well-identified restricted region of uncertainty. In Section Four, we turn our attention to Equation (\ref{Eq-Ricker4}) and employ the embedding technique to investigate the $2$-periodic case, ultimately establishing the global stability of the $2$-periodic solution under certain sufficient conditions.
 This paper closes with a conclusion section that not only summarizes our key findings but also raises pertinent open questions for further exploration.

\section{The embedding technique }
Embedding a dynamical system in general into a higher dimensional dynamical system that can be utilized to classify certain characteristics of the original system is a known approach \cite{Go-Ha1994,Sm2006,Sm2008}. However, the challenge arises when it comes to the technical details of the approach, particularly with the partial order in the embedded system.  In this section, we build the basic machinery needed in the sequel.  Let $V$ denote a partially ordered metric space, specifically the positive orthant, as defined for the purposes of this discussion, $\bbr_+^n$ or $[a,b]^n$ for some $n$.
A recursive sequence in $V$, with delay $k$, is any sequence defined by
\begin{equation}\label{recursive}
\alpha_F : x_{n+1} = F(x_{n},\ldots, x_{n-k+1}), \ \ k\geq -1,
\end{equation}
where $F:V^{k}\rightarrow V$ is a continuous function and the initial terms $x_0,x_{-1},\ldots, x_{-k+1}$ are given in $V$. These initial terms, together with $F$, determine the sequence uniquely.
We write $\mathcal S(V)$, the set of all recursive sequences in $V$. This can be topologized as a subspace of
$V^w$, the infinite direct product of $V$. Notice that different functions $F$ can give rise to the same recursive sequence \eqref{recursive}, so only the sequence uniquely determines the system.

\begin{definition} A continuous injection $\Psi: \mathcal S(V_1)\hookrightarrow\mathcal S (V_2)$, which sends convergent
sequences to convergent sequences, is called an ``embedding''. If $V_1$ is a subspace of $V_2$, then
there is a canonical inclusion $\mathcal S (V_1)\subset\mathcal S (V_2)$.
\end{definition}
It is pertinent to note that, generally, the convergence of $\Psi (\alpha)$ does not entail the convergence of $\alpha$ for a given embedding $\Psi: \mathcal S(V_1)\rightarrow \mathcal S(V_2)$.

\subsection{Embedding in the autonomous case}\label{autonomous}
Consider the two-dimensional difference equation
$\alpha_F: x_{n+1}=F(x_n,x_{n-1})$, where $F:\bbr^2_+\rightarrow\bbr$ is any continuous function.  We define $T(x,y) = (F(x,y),x)$, so that
$T(x_n,x_{n-1})=(x_{n+1},x_n)$ and its iterations produce a sequence in the plane
$$(X_n) : (x_0,x_1), (x_1,x_0),\ldots, X_n:= (x_n,x_{n-1}),$$
which entirely describes
the dynamics of the system.  We refer to  $T$ as the ``vector form'' of the system $\alpha_F$.
The system $T$ has a global attractor, which means that the sequence $(X_n)$ converges independently of the choice of the initial value within the given domain.
The assignment  $(x_n)\mapsto (X_n)$
gives an embedding $\Psi: \mathcal S(\bbr)\hookrightarrow
\mathcal S(\bbr^2)$. This is a \textit{strong} embedding in the sense that $\Psi(\alpha)$ is convergent if and only if
$\alpha$ is convergent.
\\

Next, we write $V=\bbr^2_+$. For every sequence $\alpha_F\in\mathcal S(V )$,  we define a $4$-dimensional sequence $\zeta_F\in \mathcal S(V^4)$ with general
term $\zeta_n = (X_n,X_n)= (x_n,x_{n-1},x_{n},x_{n-1})$, and initial term $\zeta_0 = (x_0,x_{-1},x_0,x_{-1})$. This is the image of $(X_n)$ under the diagonal embedding
$V^2\rightarrow V^2\times V^2$, $(x,y)\mapsto (x,y,x,y)$.
This is a strong embedding as well $\mathcal S(V)\hookrightarrow \mathcal S(V^4)$, which we refer to as the ``diagonal embedding''. Because this is a strong embedding, a good way to establish convergence of $\zeta_F$ will help us obtain convergence in our initial system.

Let us write $\zeta_F$ in recursive form as follows. Define the self-map of $V^4$ as
\begin{equation}\label{mapg}
G(x,y,u,v) = (F(x,y),u, F(u,v), x).
\end{equation}
Starting with $\zeta_0$, we see that
$\zeta_n = G^n(\zeta_0)$.
The advantage of introducing the map $G$ is to obtain a form of monotonicity \cite{Sm2008,Sm2006}.
Introduce the ``southeast partial ordering" on $V\times V$ by $(x_1,y_1)\leq_{se} (x_2,y_2)$ if and only if $x_1\leq x_2$ and $y_1\geq y_2$.  Using the southeast partial ordering on $V^2\times V^2$ means using $\leq_{se}$ on $W\times W$ with $W=V^2$.
A key observation now is that if $F$ is non-decreasing in its first component and non-increasing in its second one, i.e., $F(\uparrow, \downarrow)$, then the sequence $G$ is monotonic with respect to the $\leq_{se}$ partial ordering on $V^2\times V^2$. More precisely, rewrite  $G:V^4\rightarrow V^4$ as a map
$G: V^2\times V^2\rightarrow V^2\times V^2$, $(X,U)\longmapsto G(X,U)$. Then
$$(X_1,U_1)\leq_{se} (X_2,U_2)\ \Longrightarrow
G(X_1,U_1)\leq_{se}G(X_2,U_2).$$
This monotonicity gives a handy way of studying the convergence of the diagonal sequence $\zeta_F:\  \zeta_n=G^n(\zeta_0)$, and thus ultimately that of $\alpha_F$.  We develop this approach next.

\begin{definition}\label{box} Write $A=(a,b)$ and $B=(b,a)$. We say that $X=(x,y)$ is inside a box with vertices $A$ and $B$ in $V^2$ if $a\leq b$ and $A\leq_{se} X\leq_{se} B$. Equivalently, if in $V^4$ we have
$$(A,B)\leq_{se} (X,X)\leq_{se}(B,A)$$
\end{definition}

Suppose that for a choice of the map $F$ and the points $A$ and $B$, the map $G$ in Eq. \eqref{mapg} has the property that
$(A,B)\leq_{se}G(A,B)$ and $G(B,A)\leq_{se} (B,A)$. This happens if $(a,b)\leq_{se} (F(a,b),F(b,a))$. Then, if $X_0$ (resp. some $X_k=T^k(X_0)$ for some $k\geq 0$)
is inside a box with vertices $A,B$,  we can see immediately that for positive $n$,
\begin{equation}\label{Eq-InequalitySadokRequest}
(A,B)\leq_{se} G^n(A,B)\leq_{se} G^n(X_0,X_0)\leq_{se} G^n(B,A)\leq_{se}(B,A),
\end{equation}
which means that our sequence $\zeta_F$ is eventually caught between the orbits of $G$ through $(A,B)$ and $(B,A)$. The sequences $G^n(A,B)$ and $G^n(B,A)$, being bounded and monotonic in the southeast ordering, converge to fixed points of $G$.
A fixed point of $G$ must be of the form
$$\bar x=(x,y, y, x),\; \text{where}\; x = F(x,y)\;\text{and}\; y=F(y,x).$$
Such fixed points come in pairs if $x\neq y$, and so if  $(F(x,y), F(y,x)) = (x,y)$ has a unique solution,
then $x=y=\bar x$ is a fixed point of $F$. If $x\neq y,$ then $(x,y)$ and $(y,x)$  are dubbed as \textit{pseudo fixed points} of $F$ (these were called ``artificial fixed points'' in \cite{Al2022}). The above discussion leads us to the following result \cite{Sm2006,Sm2008}.

\begin{proposition}\label{Prop-GlobalStability} Let $F: V^2\rightarrow V$, where $F(\uparrow,\downarrow)$ and consider the system $x_{n+1}=F(x_n,x_{n-1})$ with an initial condition $X_0= (x_0,x_{-1})\in V^2$. Suppose there is $(a,b)\in V^2$ such that $(a,b)\leq_{se}(F(a,b),F(b,a))$ and $(a,b)\leq_{se} (x_k,x_{k-1})\leq_{se}(b,a)$ for some $k\geq 0$. If $(F(x,y), F(y,x)) = (x,y)$ has a unique solution in $V^2$, then
$\zeta_F$ converges to a global attractor in $V^4$ and $\alpha_F$ converges to a global attractor in $V^2.$ If $(F(x,y), F(y,x)) = (x,y)$ has three solutions $(x^*,y^*), (\bar y,\bar y)$ and $(y^*,x^*)$, $x^*<y^*$, then
$$x^*\leq \lim\inf \alpha_F\leq\lim\sup\alpha_F\leq y^*.$$
\end{proposition}

A  convenient way of stating this result is as follows: Recall that
$X_n=T^n(x_0,x_{-1})= (x_n,x_{n-1})$. Suppose $F(\uparrow, \downarrow)$ and $(a,b)\leq_{se}(F(a,b),F(b,a))$. Suppose that the sequence
$(X_n)$ is eventually in the box with vertices $A=(a,b)$ and $B=(b,a)$. If $G: V^4\rightarrow V^4$ has a unique fixed point, then it must be of the form $(x,x,x,x)$ and $\alpha_F$ must converge to $x$. Alternatively, the system $T$ has the global attractor $(x,x)$ in that box.

A $q$-cycle $C_q:=\{\xi_0,\xi_1,\ldots,\xi_{q-1}\}$ of the map $G$ in Eq. (\ref{mapg}) is a periodic solution of the system $\xi_{n+1}=G(\xi_n)$ with minimal period $q.$  Note that $C_q$ could be driven by a $q$-cycle $\{x_0,x_1,\ldots,x_{q-1}\}$  of $F,$ where $$\xi_0=(x_0,x_{q-1},x_0,x_{q-1})\quad \text{and}\quad \xi_j=(x_j,x_{j-1},x_j,x_{j-1}),\; j=1,\ldots, q-1.$$ Otherwise, we say that $F$ has a pseudo $q$-cycle. This notion has been introduced and classified in \cite{Al2022}.

\subsection{Periodic systems}

We consider in this section dynamical systems of the form
\begin{equation}\label{sequence2}
\alpha : x_{n+1} = F_n(x_n,x_{n-1})\ ,\ (x_0,x_{-1})\in V^2,
\end{equation}
where we assume $F_{n+p}=F_n$, $\forall n\geq 0$, and $p\geq 1$ is the minimal positive integer with such a property. We refer to such a sequence as being ``$p$-periodic''. These sequences are treated in \cite{Al2022}, and we give a general overview next.
We stress that the starting time $n=n_0$ is crucial in non-autonomous systems because it dictates certain synchronization between time and space. In this paper, we consider $n_0=0$ throughout.   The sequence $\alpha := \{(x_n)\}_{n\geq 0}$ in $V$ gives rise to the sequence $(X_n):= \{(x_n,x_{n-1})\}_{n\geq 0}$ in $V^2$ as before. To stress the role of the individual maps in the system (\ref{sequence2}), we denote it by $[F_0,F_1,\ldots,F_{p-1}]$ and its vector form by $[T_0,T_1,\ldots,T_{p-1}],$ where
$T_n(x_n,x_{n-1}):= (F_n(x,y), x)\ ,\ n\geq 0.$
Considering first the period two case, let us write $T_{01}=T_0\circ T_1$ and $T_{10}=T_1\circ T_0$ the compositions whose associated sequences in $V^2$ are given by
$$(X_{2n+1}): (x_{2n+1},x_{2n}) = T_{01}(x_{2n-1},x_{2n-2})\ \ \hbox{and}\ \ (X_{2n}): (x_{2n},x_{2n-1}) = T_{10}(x_{2n-2},x_{2n-3})$$ with initial values  $(x_0,x_{-1})$ and $(x_1,x_0)$ respectively.
These two sequences are interlocked as in Figure \ref{Fig-Diagram}.

\begin{figure}[htb]
\centering
\begin{tikzpicture}
\node (A) at (3,0) {$(x_1,x_0)$};
\node (B) at (9,0) {$(x_3,x_2)$};
\node (C) at (0,-1.5) {$(x_0,x_{-1})$};
\node (D) at (6,-1.5) {$(x_2,x_1)$};
\node (E) at (10,-1.5) {};
\draw[->] (A) -- node[above]{$T_{01}$} (B);
\draw[->] (C) -- node[below]{$T_{10}$} (D);
\draw[->] (D) -- node[below]{$T_{10}$} (E);
\draw[dashed,->] (C) -- node[left]{$T_0$} (A);
\draw[dashed,->] (A) -- node[left]{$T_1$} (D);
\draw[dashed,->] (D) -- node[right]{$T_0$} (B);
\end{tikzpicture}
\caption{}\label{Fig-Diagram}
\end{figure}

Decomposing the $2$-periodic recursive sequence $(X_n)$ into these two subsequences give an embedding
$\mathcal S(V)\rightarrow\mathcal S(V^2)\times\mathcal S(V^2)$ whose image is
$(X_{2n+1},X_{2n})$. Concatenating these sequences side-by-side produces a recursive sequence $\zeta = \Psi (\alpha)$ in $\mathcal S(V^4)$ given by the general term
$$\zeta_n:= (x_{2n+1}, x_{2n}, x_{2n},x_{2n-1})\; \text{where}\; \zeta_0 = (x_1,x_0, x_0,x_{-1}).$$

The above is not a strong embedding in general since a convergence of a pair of subsequences implies convergence of the original sequence unless they both converge to the same limit. In all cases, this construction extends in an obvious manner to $p$-periodic recursive sequences of finite delay and gives an embedding of the subset of recursive $p$-periodic sequences of delay $N$ into $\mathcal S(V^{N})$.

Starting with Eq. \eqref{sequence2}, $p=2$, the system in vector form is written $[T_0,T_1]$, and it breaks into the two ``folded'' systems $T_{10}$ and $T_{01}$ as already indicated.
We will relate the $k$-cycles of $T_{10}$ and $T_{01}$ to the $2k$-cycles of $[T_0,T_1]$. Note that this process is known as folding and unfolding, and it has been characterized for one-dimensional maps \cite{Al-Ca-Li2014a}.

\begin{lemma}\label{2periodic}
 Define $\alpha : x_{n+1} = F_n(x_n,x_{n-1}), (x_0,x_{-1})\in V^2$ which is $2$-periodic.
 Let $C_k(T_{01})\subset (V^2)^k$ (resp. $C_k(T_{10})$) describe the set of $k$-cycles of $T_{01}$ (resp. $T_{10})$. Then $T_0$ maps $C_k(T_{01})$ bijectively onto $C_k(T_{10})$, with inverse $T_1$. If these sets are disjointed, their union is a $2k$-cycle of $[T_0,T_1]$.  In particular, $[T_0,T_1]$ has a unique $2$-cycle if and only if $T_{01}$ and $T_{10}$ have unique fixed points which are distinct.\end{lemma}

\begin{proof}
   This is an immediate consequence of the identities $T_1\circ T_{01} = T_{10}\circ T_1$, and $T_0\circ T_{10} = T_{01}\circ T_0$.
The final claim is a consequence of the fact that the fixed points of $\alpha$ correspond to its one cycle.
\end{proof}

 In the next two illustrative examples, we ignore monotonicity, and focus on the structure of cycles.

\begin{example}\label{first}\rm Let $p=k=2$.
Here's an example where the $2$-cycles for $T_{01}$ and $T_{10}$ are different.
Define $\alpha$ as in Eq. \eqref{sequence2} with
$F_0(x,y)= y+x^2-1$ and $F_1(x,y)= -y$. Then
$$T_0(x,y) = (y+x^2-1, x)\; \text{and}\; \ T_1(x,y) = (-y,x).$$
We can check that $\{(1,y), (-1,y)\}$ is the unique $2$-cycle of $T_{10}=T_1T_0$, while $\{(y,1), (y,-1)\}$ is the unique $2$-cycle of $T_{01}=T_0T_1$. Clearly,
$T_0(1,y)=(y,1)$ and $T_1(y,1)=(-1,y)$. The system $[T_0,T_1]$ has a $4$-cycle.
\end{example}

\begin{example}\label{second}\rm  Define $\alpha$ as in \eqref{sequence2} with
$F_0(x,y) = xy$ and $F_1(x,y)={x\over y}$. This is $2$-periodic, with subsequences in vector form induced from $T_0(x,y) = (xy,x)$ and $T_1(x,y) = ({x\over y},x)$. A  calculation shows that both systems  $T_{10}$ and $T_{01}$ have each a unique three cycle that is  common to both
$$\{(-1,-1), (1,-1), (-1,1)\} =  C_3(T_{10})=  C_3(T_{01}).$$
This is also the same $3$-cycle for each individual map $T_j$ and for the $2$-periodic system $[T_0,T_1].$ This is due to the fact that the periodicity of the system $2$ and the periodicity of the cycle $3$ are coprime numbers.
\end{example}

Similarly, treating the $p$-periodic case for any $p\geq 2$ is possible. In that case, the sequence $\alpha$ in \eqref{sequence2} gives rise to $p$ sequences in $V^2$, indexed by the action of the cyclic group $\mathbb Z_p$ with generator $\sigma$ acting on
the ordered tuple $(p-1,p-2,\ldots, 1,0)$  by cyclic permutations of coordinates. Here $\sigma^i(p-1,p-2,\ldots, 1,0) = (i-1,i-2\ldots, 0,\ldots, i)$.
To this corresponds the sequence
$X_{\sigma^i}$ with general recursive term
\begin{equation}\label{vecsys}
(x_{n+1},x_{n}) = T_{\sigma^i}(x_n,x_{n-1})\ \ \hbox{where}\ \ T_{\sigma^i}=T_{i-1} T_{i-2}\cdots T_0\cdots T_i,
\end{equation}
$T_i(x,y)=(F_i(x,y), x)$, with initial term $T_{i-1}\cdots T_0(x_0,x_{-1})$.
The sequences $X_{\sigma^i}$ partition $(X_n)$. They are mirror copies of each other in the sense that either all converge simultaneously or diverge. If some sequence $X_{\sigma^i}$ has a $k$-cycle, then all other sequences have $k$-cycles. If some sequence has a global attractor, so do all other sequences (see \cite{Al-Ca-Li2014a}, Lemma 2.1).
It is of potential interest to understand how $k$-cycles of the $T_{\sigma^i}$s' produce $q$-cycles of $[T_0,T_1,\ldots, T_{p-1}]$. The $k$-cycles in $C_k(T_{\sigma^i})$, for different $i$, may overlap or even be equal as illustrated in Example \ref{first} and Example \ref{second}. Note for example that if the $F_i$'s are injective, then a fixed point of $T_{\sigma^i}$ gives rise to a $p$-cycle of $[T_0,T_1,\ldots, T_{p-1}]$. This notion is a straightforward generalization of the one-dimensional case in \cite{Al-Ca-Li2014a}.

\subsection{Embedding in the periodic case} \label{SubSection-2.3}
In the autonomous case (Subsection \ref{autonomous}),
we extended a one dimensional system $\alpha_F$ as given in Eq. \eqref{recursive} into a
a four dimensional recursive system $\zeta_F : \zeta_{n+1} = G(\zeta_n)$, where $G$ is a self map of $V^4$. When working in a box (see Definition \ref{box}),  if $G$ had a unique fixed point, then $\alpha_F$ converged to a unique fixed point. We seek an analogous result in the periodic case, i.e., Eq. \eqref{sequence2}. It turns out in this case that a unique fixed point of some corresponding embedded system in $V^4$ leads to the existence of a globally attracting cycle of $[T_0,T_1]$ (and $\alpha_F$). We will restrict ourselves below to the $2$-periodic case, i.e., $p=2$.

Starting with \eqref{sequence2}, $F_0 , F_1$, $F_1\neq F_0$ and $F_i(\uparrow,\downarrow)$, define similarly
$$G_0(x,y,u,v) = (F_0(x,y),u, F_0(u,v), x)\ \ ,\ \ G_1(x,y,u,v) = (F_1(x,y),u, F_1(u,v), x).$$
We write $G_{01} = G_0\circ G_1$ and $G_{10}= G_1\circ G_0$.  The four dimensional embedded system $[G_0,G_1]$ is defined by  $\zeta_{n+1} = G_{n\mod 2}(\zeta_n)$. Again, the restriction of $[G_0,G_1]$ to the diagonal subset in $V^2\times V^2$ is the ``diagonal system''
$([T_0,T_1],[T_0,T_1])$. The following is a direct check.

\begin{lemma} \label{onetoone}
Assume that $F_i(\uparrow,\downarrow)$ are \textit{strictly} monotonic in each component. Then both $G_0$ and $G_1$ are one-to-one.
\end{lemma}

A fixed point of the system $[G_0,G_1]$ must be a fixed point of all the $G_i$'s. In other words, if $\bar\zeta\in V^4$ is a fixed point of $[G_0,G_1]$, then
\begin{equation}\label{formfixed}\bar{\zeta} = (\bx,\by,\by,\bx), \; \text{where}\;
F_j(\bx, \by) = \bx\; \text{and}\; F_j(\by,\bx) = \by, j=0,1.
\end{equation}
Note that a fixed point of
$[G_0,G_1]$ is necessarily a fixed point of both $G_{01}$ and $G_{10}$. The converse is not always true. It could happen that $G_{01}$ and $G_{01}$ have different fixed points that form a $2$ cycle of $[G_0,G_1]$. The relationship between the cycles of $[F_0,F_1]$ and $[G_0,G_1]$ turns out to be interesting. It has been characterized in \cite{Al2022}, but we illustrate it here for clarity and completeness.  A $2$-cycle $\{x_0,x_1\}$ of $[F_0,F_1]$ must satisfy
 \begin{equation}\label{2cycles}
F_0(x_1,x_0)=x_0\ \hbox{and}\ F_1(x_0,x_1)=x_1,\quad x_1\neq x_0.
\end{equation}
A $2$-cycle of $[G_0,G_1]$ is of the form $\{\zeta_0,\zeta_1\}\in V^4\times V^4$ such that
$G_0(\zeta_0)=\zeta_1$ and $G_1(\zeta_1)=\zeta_0$, $\zeta_0\neq\zeta_1$.
Let $\zeta = (\bar x,\bar y,\bar u,\bar v)$ be a fixed point of $G_{10}$, that
is  $G_1(G_0(\zeta)) = \zeta$.
Then we must have
\begin{equation}\label{condition}
(\bar u,\bar v) = (F_1(\bar y,\bar x), F_0(\bar x,\bar y))\ \ \hbox{and}\ \ (\bar x,\bar y)=
(F_1(\bar v,\bar u), F_0(\bar u,\bar v)).
\end{equation}
Note that if $\zeta = (\bar x,\bar y,\bar u,\bar v)$ is a fixed point of $G_{10}$, then
$\zeta' = (\bar u, \bar v, \bar x, \bar y)$ is also a fixed point. A unique fixed point of
$G_{10}$ is therefore of the form $(\bar x,\bar x,\bar x,\bar x)$ or
$(\bar x,\bar y,\bar x,\bar y)$, $\bar x\neq \bar y$ and satisfying the conditions \eqref{condition}.
We discuss next the various implications of the existence of the fixed points of $G_{10}$ on
the cycles of $[G_0,G_1]$ and $[F_0,F_1].$  Assume $\xi_0 = (\bar x,\bar y,\bar u,\bar v)$ satisfies the system of equations in \eqref{condition}.
\begin{description}
\item{(i)} Let $\bar x=\bar y.$ We have the following scenarios:
\begin{itemize}
\item $\bar \xi_0=(\bar x,\bar x,\bar x,\bar x)$ and $\bar x$ is a common fixed point for each individual map $F_j.$ Obviously, $\bar x$ will be a fixed point of $[F_0,F_1].$
\item  $\bar \xi_0=(\bar x,\bar x,\bar u,\bar u), $ and $G_0(\bar \xi_0)=\bar \xi_1=(\bar u,\bar u,\bar x,\bar x), $ where $\bar x\neq \bar u.$ In this case, $\{\bar \xi_0,\bar \xi_1\}$ is a $2$-cycle of $[G_0,G_1]$ and $\{\bar x,\bar u\}$ is a $2$-cycle of the one dimensional $2$-periodic system $[f_0,f_1],$ where $f_0(t)=F_0(t,t)$ and $f_1(t)=F_1(t,t).$ $\{(\bar x,\bar x),(\bar u,\bar u)\}$ is dubbed as a pseudo $2$-cycle of $[F_0,F_1].$
\end{itemize}
\item{(ii)} Let $\bar x\neq \bar y.$ We have the following scenarios:
\begin{itemize}
\item $\bar\xi_0=(\bar x,\bar y,\bar y,\bar x),$ where $\bar\xi_0$ is a common fixed point of each individual map $G_j.$ In this case, $(\bar x,\bar y)$ and $(\bar y,\bar x)$ are dubbed as two pseudo common fixed points of each map $F_j.$
\item $\bar \xi_0=(\bar x,\bar y,\bar x,\bar y) $ and $G_0(\bar \xi_0)=\bar \xi_1=(\bar y,\bar x,\bar y,\bar x). $  In this case, $\{\bar \xi_0,\bar \xi_1\}$ is a $2$-cycle of $[G_0,G_1]$ and $\{\bar y,\bar x\}$ is a $2$-cycle of $[F_0,F_1].$
\item $\bar \xi_0=(\bar x,\bar y,\bar u,\bar v)$ and $ \bar \xi_1=(\bar v,\bar u,\bar y,\bar x),$ where $(\bar u,\bar v)$ is neither $(\bar x,\bar y)$ nor $(\bar y,\bar x).$  In this case, $(\bar u,\bar v)=(F_1(\bar y,\bar x),F_0(\bar x,\bar y))$ and $C_2:=\{\bar \xi_0,\bar \xi_1\}$ is a $2$-cycle of $[G_1,G_0],$ while $\{(\bar x,\bar y),(\bar v,\bar u)\}$, $\{(\bar u,\bar v),(\bar y,\bar x)\}$ are dubbed as pseudo $2$-cycles of $[F_0,F_1].$
\end{itemize}
\end{description}

Next, we need to focus on stability. Given a unique fixed point of $G_{01}$, and provided we are working ``in a box'', we deduce that
$T_{01}$ also has a unique fixed point. Similarly for $T_{10}$. By Lemma \ref{2periodic}, $[T_0,T_1]$ has a $2$-cycle if the fixed point is not common.
This must be a globally attracting cycle \cite{Al2022}, and we extract the needed result below.

\begin{proposition}\label{Pr-GSConditions}
Let $F_j: \mathbb{R}_+^2\rightarrow \mathbb{R}_+$, where $F_j(\uparrow,\downarrow)$ for each $j=0,1.$ Consider the system $[F_0,F_1]$ with initial conditions $X_0= (x_0,x_{-1})\in \mathbb{R}_+^2$. Suppose there exists $(a,b)\in \mathbb{R}_+^2$ such that $a<b,$ $(a,b)\leq_{se} (x_k,x_{k-1})\leq_{se}(b,a)$ for some $k\geq 0$ and
\begin{equation}\label{In-ConditionsOnAandB}
a< \min\{F_0(a,b), F_1(F_0(a,b),b)\},\; b> \max\{F_0(b,a),F_1(F_0(b,a),a)\}.
\end{equation}
 If $G_{10}=G_1\circ G_0$ has a unique fixed point in $\mathbb{R}_+^4$, then the $2$-periodic system
$[F_0,F_1]$ has a global attractor, which can be an equilibrium solution or a $2$-cycle.
\end{proposition}
\begin{proof}
The condition
$(a,b)\leq_{se} (x_k,x_{k-1})\leq_{se} (b,a)$ means that the sequence $(X_n)$ is eventually trapped in the box $[a,b]^2$. The condition \eqref{In-ConditionsOnAandB} gives us
$(A,B)< G_{10}(A,B)$, where $A=(a,b)$ and $B=(b,a)$. Since $G_{10}$ is monotonic with respect to $\leq_{se}$, we can use our formalism directly (Section 2.1, inequalities \eqref{Eq-InequalitySadokRequest}) to infer that the sequence $G_{10}^n(X_0,X_0)$ converges to the unique fixed point of $G_{10}$ in $[a,b]^4$. If the unique fixed point of $G_{10}$ is of the form $\bar \xi=(\bar x,\bar x,\bar x,\bar x),$ then $\bar x$ is an equilibrium point of $[F_0,F_1]$ and both $\{x_{2n}\},\{x_{2n-1}\}$ converge to $\bar x.$ On the other hand, if the unique fixed point of $G_{10}$ is of the form $\bar \xi=(\bar x,\bar y,\bar x,\bar y),\; \bar x\neq \bar y$ then $\{\bar x,\bar y\}$ is a $2$-cycle of $[F_0,F_1],$  and the even terms of the orbit $\{x_n\}$ converge to one branch of the $2$-cycle, while the odd terms converge to the other branch.   Hence, whether the system $[F_0,F_1]$ has an equilibrium or a $2$-cycle, it serves as a global attractor.
\end{proof}
We remark that the inequalities (\ref{In-ConditionsOnAandB}) can be replaced by the more simple inequalities $(a,b)<_{se}(F_j(a,b),F_j(b,a)),\;j=0,1.$ However, simplicity here comes at the expense of generalization.


\section{Stability under constant stocking}

Since $h>0,$ all solutions of Eq. (\ref{Eq-Ricker3}) must satisfy $y_n>h$ for all $n.$ Furthermore, an equilibrium solution $\bar y$  satisfies $\left(1-\frac{h}{\bar y}\right)=e^{r-\bar y},$ and it is obvious that we obtain a unique equilibrium solution. To stress the role of the parameter $h$ in our analysis, we denote the positive equilibrium solution by $\bar y_h.$  The Jacobian matrix at $\bar y_h$ is given by
$$J(\bar y_h,\bar y_h)=\left[
                         \begin{array}{cc}
                           e^{r-\bar y_h} & -\bar y_h e^{r-\bar y_h} \\
                           1 & 0 \\
                         \end{array}
                       \right]=\left[
                         \begin{array}{cc}
                           1-\frac{h}{\bar y_h} & h-\bar y_h \\
                           1 & 0 \\
                         \end{array}
                       \right].
$$
Let the trace of the Jacobian matrix be $T$ and the determinant be $D,$ then $0<T=1-\frac{h}{\bar y_h}<1$ and $D=\bar y_h-h>0.$ Furthermore, we have
$$D-T=\bar y_h-h+\frac{h}{\bar y_h}-1>-1.$$
Based on the Jury conditions for stability, the eigenvalues of $J$ are within the unit disk when $|T|<1+D<2$, and consequently, the equilibrium $\bar y_h$ is locally asymptotically stable (LAS) when $D<1,$ i.e., $\bar y_h<1+h.$ Since
$$ 0<e^{r-{\bar y_h}}=1-\frac{h}{\bar y_h}<1,$$
we obtain $\bar y_h>h$ and $\bar y_h > r.$
In fact, $\bar y_0=r.$ The next lemma is relevant in the sequel and describes the behavior of the fixed point in terms of $r$ and $h$.

\begin{lemma}\label{increasing}
    The fixed point $\bar y_h$ of Eq. (\ref{Eq-Ricker3}) is increasing in both $h$ and $r$.
\end{lemma}

\begin{proof}
This is an immediate consequence of implicitly differentiating
$\displaystyle\left(1-\frac{h}{\bar y_h}\right)=\frac{d}{dh} e^{r-{\bar y_h}}$
with respect to both $h$ and $r$. In the first case, we obtain
\begin{align*}
\left(\bar y_h-h+\frac{h}{\bar y_h}\right)\frac{d\bar y_h}{dh}=1.
\end{align*}
So that $\displaystyle \frac{d\bar y_h}{dh}>0$. Similarly, $\displaystyle \frac{d\bar y_h}{dr} > 0$ and $\bar y_h$ is increasing in $r.$
\end{proof}

Now, the question is: can $\bar y_h$ reach $h+1$ and lose its stability? If yes, then we must have $h+1=e^{h+1-r}.$ This is possible when
\begin{equation}\label{Eq-rh}
r=r_1:=h+1-\ln(h+1).
\end{equation}

\begin{lemma}\label{stable-unstable}
Consider Eq. (\ref{Eq-Ricker3}) with $h>0,$ and let $r_1$ be as defined in Eq. (\ref{Eq-rh}).  Each of the following holds true:
\begin{description}
\item{(i)} Eq. (\ref{Eq-Ricker3}) has a unique non-negative equilibrium solution $\bar y_h$ and $\bar y_h>\max\{r,h\}.$
\item{(ii)}  If $0<r<r_1,$ then $\bar y_h$ is LAS, while   if $r>r_1,$ then $\bar y_h$ is unstable.
\end{description}
\end{lemma}

\begin{proof}
   Part (i) is clear. Part (ii) remains to be clarified. Recall that $\bar y_h$ is LAS if and only if $\bar y_h<h+1$; in turn, we claim this happens iff $r<r_1$. But $h+1$ is the value of $\bar y_h$ when $r=r_1$, and so the claim is a consequence of the fact that $\bar y_h$ is increasing in $r.$  When $r=r_1$, the claim and its proof are standard.
\end{proof}
When $r=r_1,$ Eq. (\ref{Eq-Ricker3}) has the potential to go through a Neimark-Sacker bifurcation (see Fig. \ref{Fig-NeimarkSacker1}).
Note that since $r_1$ is increasing in $h,$ $0<r<r_0=1$ gives LAS regardless of the value of $h.$
Our next example illustrates the validity of Part (ii) of Lemma \eqref{stable-unstable}. Also, Fig. \ref{Fig-StabilityRegions1} shows the stability region in the $(h,r)$-plane.

\begin{example}
Let $h=e-1$ and $r=2.$
Then $r_1=e-1 < 2=r$. In this case,
we obtain $\bar y_h\approx 2.898,$  and the eigenvalues of the Jacobian matrix are $\lambda\approx 0.204 \pm 1.067i.$ They are out of the unit disk, and consequently, the equilibrium is unstable. On the other hand, if we consider $h=e-1$ and $r=\frac{3}{2},$ then $r_1=e-1>r$, $\bar y_h\approx 2.589$ and, the eigenvalues of the Jacobian matrix are $\lambda\approx 0.168\pm 0.918i.$ The eigenvalues are located within the unit disk; consequently, the equilibrium is LAS.
\end{example}
\begin{figure}[htpb]
    \centering
    \includegraphics[width=0.6\textwidth]{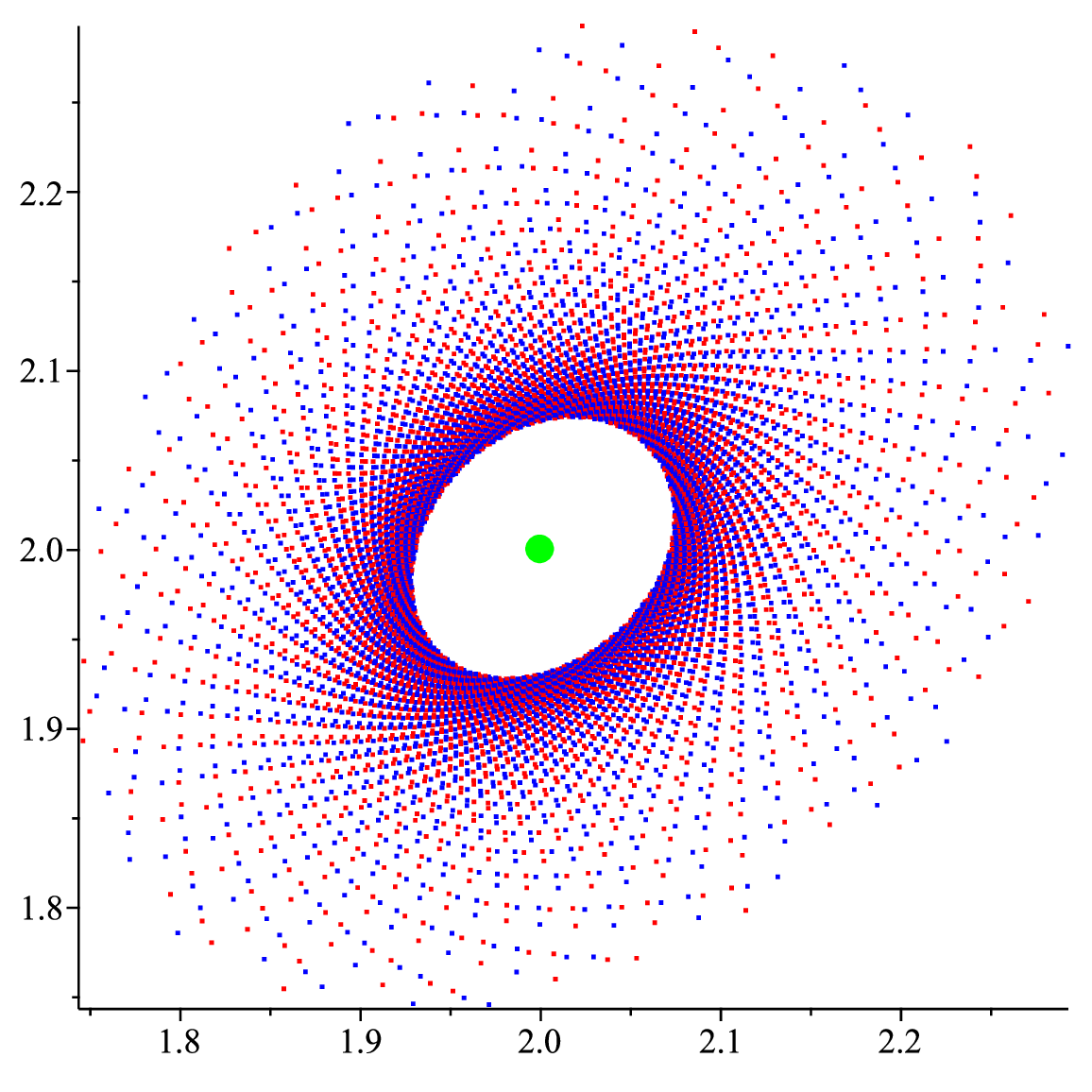}
    \caption{Let $r=r_1.$ This figure shows the equilibrium solution $y_h$ (solid green circle) undergoes a Neimark-Sacker bifurcation at $h=1.0.$  The red color represents the even terms in the orbits, while the blue represents the odd terms. The horizontal axis is for $x_{n}$, and the vertical axis is for $x_{n-1}.$}
    \label{Fig-NeimarkSacker1}
\end{figure}

Next, we define the $4$-dimensional map
$$G:\; \mathbb{R}_+^4\;\to \mathbb{R}_+^4, \quad \text{where} \quad G(x,y,u,v)=(xf(y)+h,u,uf(v)+h,x).$$
Again here, $f(y)=e^{r-y}$. We solve $G(\xi)=\xi $ to find the fixed points of $G.$ We obtain $\xi=(x^*,y^*,y^*,x^*),$ where $x^*=x^*f(y^*)+h$ and $y^*=y^*f(x^*)+h$, with $x^*\leq y^*$. The case $x^*=y^*$ is when both are equal $\bar y_h,$ and we write the symmetric solution $\bar \xi=\bar \xi_h=(\bar y_h,\bar y_h,\bar y_h,\bar y_h).$

Let $L_1$ and $L_2$ denote the curves of $y=yf(x)+h$ and $x=xf(y)+h,$ respectively. We discuss how these curves intersect since the intersection points (i.e. the ``solutions'') correspond to the fixed points of $G$. The fixed point $(\bar y_h,\bar y_h)$ is always a solution. Clearly, $L_1$ has a vertical asymptote at $x=r$ and a horizontal asymptote at $y=h,$ whereas $L_2$ has a vertical asymptote at $x=h$ and a horizontal asymptote at $y=r.$  A variation of $h$ makes asymmetric solutions bifurcate from the symmetric ones. The slope $\displaystyle {dy\over dx}$ of $L_1$ at $\bar y_h$ is $\frac{1}{h} \bar y_h(h-\bar y_h),$ and the reciprocal of this expression is the slope of $L_2.$ Since the slope is negative and $\bar y_h>h,$ we consider the solution of  $\bar y_h(\bar y_h-h)=h$ to be
\begin{equation}\label{Eq-H*}
H^*:=\frac{1}{2}\left(h+\sqrt{h^2+4h}\right).
\end{equation}
In this case,
\begin{equation}\label{Eq-r2}
    r=r_2:=H^*+\ln (H^*-h)-\ln(H^*),\ \ \
    \ (\hbox{and}\ \bar y_h=H^*).
\end{equation}
Notice that we always have (we skip the computation) that
\begin{equation}\label{r1r2h}
r_2 < r_1 \quad \text{and{\quad $r_1<h$\quad\text{if}\quad $h > e-1$}}.
\end{equation}
To recapitulate, for a given $h$, and for
$r=r_2$, the fixed point of Eq. (\ref{Eq-Ricker3})
is $\bar y_h = H^*$, and at the solution $x=y = H^*$, both curves have slope $-1$.  The number of solutions of $L_1\cap L_2$ depends on whether $\bar y_h$ is larger or smaller than $H^*$, but this also depends on whether $h<r$ or $h>r$. From a combinatorial standpoint, there are four options; however,  $(h,H^*)<_{se}(r,\bar y_h)$ is not viable according to the following result.

\begin{lemma}\label{hlessr}
    If $h\leq r,$ then $H^*\leq \bar y_h.$
\end{lemma}

\begin{proof}
Because $1+t\leq e^t$ for all $t\in \mathbb{R},$ then $1+h-\bar y_h\leq e^{h-\bar y_h}.$
Since $h\leq r,$ we obtain that
$1+h-\bar y_h\leq e^{r-\bar y_h},$
and consequently,
$$1-e^{r-\bar y_h}\leq \bar y_h-h \quad \Longleftrightarrow\quad \bar y_h(1-e^{r-\bar y_h})\leq \bar y_h(\bar y_h-h) .$$
From the fact that $\bar y_h$ is an equilibrium point, we obtain $h=\bar y_h(1-e^{r-\bar y_h}),$
and therefore, $h\leq \bar y_h(\bar y_h-h),$ which gives us $\bar y_h\geq H^*.$
\end{proof}

The following lemma summarizes and settles all the cases.

\begin{lemma}\label{Lem-Intersections}
Consider $r_2$ as defined in Eq. (\ref{Eq-r2}). Let $L_1$ and $L_2$ denote the curves of $y=F(y,x)=yf(x)+h$ and $x=F(x,y)=xf(y)+h,$ respectively. Each of the following holds true:
\begin{description}
\item{(i)} If $h<r,$ then $H^*<\bar y_h$ and $L_1$ and $L_2$ intersect at the unique point $(\bar y_h,\bar y_h).$
\item{(ii)} If $r< r_2 < h,$ then $\bar y_h<  H^*$ and $L_1$ and $L_2$ intersect at the unique point $(\bar y_h,\bar y_h).$
\item{(iii)} If $r_2<r<h,$ then $H^*< \bar y_h$ and $L_1$ and $L_2$ intersect at three  points denoted by $(\bar y_h,\bar y_h), (x^*,y^*) $ and $ (y^*,x^*),$ where $x^*\neq y^*.$
\end{description}
\end{lemma}

\begin{proof} (i) The two curves $L_1$ and $L_2$
always intersect at the point $(\bar y_h,\bar y_h)$ along the diagonal.
Since the curves are symmetric images of one another with respect to the $y=x$ axis, we can confine ourselves to the case
$x> \bar y_h> \max\{r,h\}$ and investigate other intersection points. Express $y$ explicitly for both $L_1$ and $L_2$ to find, respectively, that
$$y_1 = {h\over 1-e^{r-x}}\quad \text{and}\quad
y_2 = r-\ln \left(1- {h\over x}\right).$$
We show $y_2-y_1>0$ for all $x>\bar y_h$  when $h<r.$ For a fixed value of $x,$ we show that $y_2-y_1$ is positive and minimum at $h=r.$ We have
$${d\over dh}(y_2-y_1) = {-1\over 1-e^{r-x}} + {1\over x-h}$$
and $e^{r-x}> 1+r-x$, for all $x>r.$
Therefore,
${1\over 1-e^{r-x}}> {1\over x-r}$, and since  $h<r<\bar y_2<x,$  we obtain
$${d\over dh}(y_2-y_1) < {-1\over x-r} + {1\over x-h}<0.$$
So $y_2-y_1$ is decreasing and takes its minimum at $h=r.$ At $h=r,$
$$y_2-y_1={-r\over 1-e^{r-x}} + r - \ln \left(1-{r\over x}\right)>0.$$
 Hence, the minimum of $y_2-y_1$ is strictly positive, and $y_2\neq y_1$.  This shows that $L_1$ and $L_2$ cannot intersect again for $x>\bar y_h> \max\{r,h\}$, and also by symmetry, for $x<\bar y_h$. This completes the proof of Part (i).

To show (ii) and (iii), we have shown in Lemma \ref{increasing} that the fixed point $\bar y_h$ is an increasing function of $r.$ Therefore, if $r<r_2<h$, then necessarily $\bar y_h<H^*$, since $H^*$ is the fixed point when $r=r_2$. This establishes the first part of the claims in (ii) and (iii). The second part in each case follows along the same outline as in (i).
\end{proof}

Figures \ref{Fig-Intersections1} and \ref{Fig-Intersections2} illustrate all three possible scenarios in Lemma \ref{Lem-Intersections}.

\definecolor{ffqqqq}{rgb}{1.,0.,0.}
\definecolor{qqqqff}{rgb}{0,0,1}
\definecolor{zzttqq}{rgb}{0.6,0.2,0}
\begin{figure}[H]
\centering
\begin{minipage}[t]{0.5\textwidth}
\raggedright
\begin{center}
\begin{tikzpicture}[line cap=round,line join=round,>=triangle 45,x=1.0cm,y=1.0cm,scale=0.5]
\draw[-triangle 45, line width=1.0pt,scale=1] (0,0) -- (10.0,0) node[below] {$x$};
\draw[line width=1.0pt,-triangle 45] (0,0) -- (-1,0);
\draw[-triangle 45, line width=1.0pt,scale=1] (0,0) -- (0,10) node[left] {$y$};
\draw[line width=1.0pt,-triangle 45] (0,0) -- (0.0,-1);
\draw[line width=1.2pt,domain=0:9.0,smooth,variable=\x,green] plot ({\x}, {\x)});
\draw[line width=1pt,color=red,smooth,samples=100,domain=2.08:10.0] plot(\x,{0.7/(1-exp(2-\x))});
\draw[line width=1pt,color=blue,smooth,samples=100,domain=0.701:10.0] plot(\x,{2-ln(1-0.7/(\x))});
\draw[scale=1] (0.7,8.5) node[above] {\footnotesize $L_2 $};
\draw[scale=1] (2.0,9.0) node[above] {\footnotesize $L_1 $};
\draw[scale=1] (0,2) node[left] {\footnotesize $r $};
\draw[scale=1] (0,0.7) node[left] {\footnotesize $h $};
\draw[line width=0.6pt,dashed,color=black] (0,0.7) -- (9.0,0.7);
\draw[line width=0.6pt,dashed,color=black] (0,2.0) -- (9.0,2.0);
\draw[line width=1.0pt,red,fill=yellow] (2.35,2.35) circle (2.0pt);
\draw[scale=1] (5,-0.3) node[below] {\footnotesize (i) $r=2, h=0.7$};
\end{tikzpicture}
\end{center}
\end{minipage}%
\begin{minipage}[t]{0.5\textwidth}
\raggedright
\begin{center}
\begin{tikzpicture}[line cap=round,line join=round,>=triangle 45,x=1.0cm,y=1.0cm,scale=0.5]
\draw[-triangle 45, line width=1.0pt,scale=1] (0,0) -- (10.0,0) node[below] {$x$};
\draw[line width=1.0pt,-triangle 45] (0,0) -- (-1,0);
\draw[-triangle 45, line width=1.0pt,scale=1] (0,0) -- (0,10) node[left] {$y$};
\draw[line width=1.0pt,-triangle 45] (0,0) -- (0.0,-1);
\draw[line width=1.2pt,domain=0:9.0,smooth,variable=\x,green] plot ({\x}, {\x)});
\draw[line width=1pt,color=red,smooth,samples=100,domain=2.4:10.0] plot(\x,{3.0/(1-exp(2-\x))});
\draw[line width=1pt,color=blue,smooth,samples=100,domain=3.004:10.0] plot(\x,{2-ln(1-3.0/(\x))});
\fill[line width=2pt,color=zzttqq,fill=zzttqq,fill opacity=0.300] (3.0,9.0) -- (3.007,8.000) -- (3.020,7.000) -- (3.056,6.003) -- (3.100,5.430) -- (3.159,4.9898) -- (3.4,4.248) -- (3.68,3.68) --(3.2,4.29)-- (3.03,4.66) -- (2.75,5.68) -- (2.6,6.64) -- (2.57,6.9) -- (2.46,8.1) -- (2.43,8.58) -- (2.405,9.0) -- (3,9) -- cycle;

\draw[scale=1] (3.2,8.5) node[above] {\footnotesize $L_2 $};
\draw[scale=1] (9.0,3.0) node[above] {\footnotesize $L_1 $};
\draw[scale=1] (0,2) node[left] {\footnotesize $r $};
\draw[scale=1] (0,3) node[left] {\footnotesize $h $};
\draw[line width=0.6pt,dashed,color=black] (0,3.0) -- (9.0,3.0);
\draw[line width=0.6pt,dashed,color=black] (0,2.0) -- (9.0,2.0);
\draw[line width=1.0pt,red,fill=yellow] (3.7,3.7) circle (2.0pt);
\draw[scale=1] (5,-0.3) node[below] {\footnotesize (ii) $r=2, h=3.0$};
\end{tikzpicture}
\end{center}
\end{minipage}%
\caption{Part (i) of this figure shows the curves of $L_1$ and $L_2$ when $H^*\leq \bar y_h$ and $h<r.$ The unique intersection point is $(\bar y_h,\bar y_h).$  Part (ii) shows the unique intersection point $(\bar y_h,\bar y_h)$ but when $\bar y_h\leq H^*$ and $h>r.$ The shaded region in Part (ii) represents the solution of $(x,y)\leq_{se} (F(x,y),F(y,x))$ in which $x\leq y.$  }\label{Fig-Intersections1}
\end{figure}
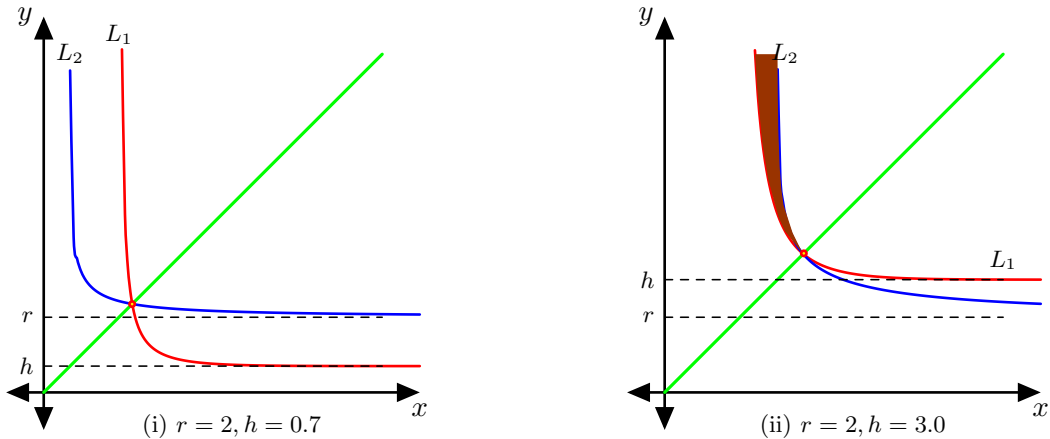

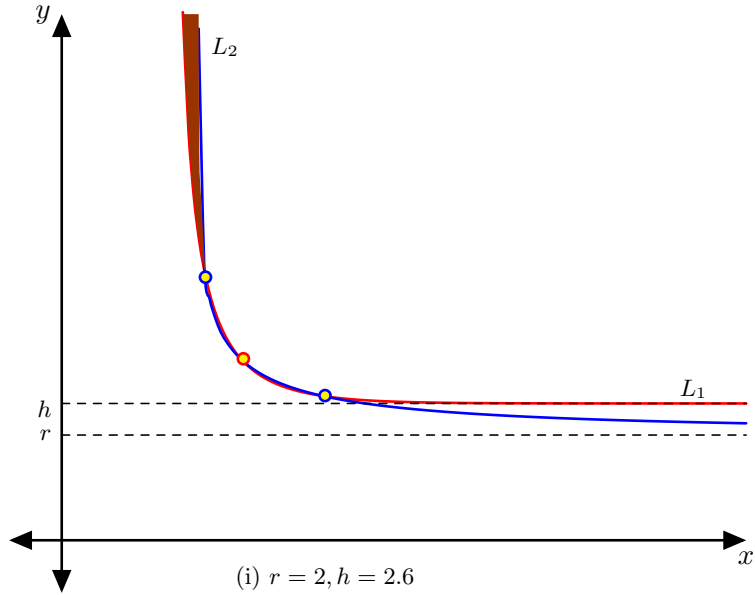
\begin{figure}[h!]
\centering
\begin{minipage}[t]{0.8\textwidth}
\raggedright
\begin{center}
\begin{tikzpicture}[line cap=round,line join=round,>=triangle 45,x=1.0cm,y=1.0cm,scale=0.7]
\draw[-triangle 45, line width=1.0pt,scale=1] (0,0) -- (13.0,0) node[below] {$x$};
\draw[line width=1.0pt,-triangle 45] (0,0) -- (-1,0);
\draw[-triangle 45, line width=1.0pt,scale=1] (0,0) -- (0,10) node[left] {$y$};
\draw[line width=1.0pt,-triangle 45] (0,0) -- (0.0,-1);
\draw[line width=1pt,color=red,smooth,samples=100,domain=2.3:13.0] plot(\x,{2.6/(1-exp(2-\x))});
\draw[line width=1pt,color=blue,smooth,samples=100,domain=2.601:13.0] plot(\x,{2-ln(1-2.6/(\x))});
\fill[line width=2pt,color=zzttqq,fill=zzttqq,fill opacity=0.300]  (2.7319,5.009) -- (2.5666,6.0108) -- (2.462467183264971,7.021866377005591) -- (2.3909833483784224,8.03439756408671) -- (2.339442791964597,9.03301683961519) -- (2.299946835151726,10.0)--(2.6008901528549897,10.0) -- (2.6024865996029796,8.953306489679896) -- (2.6065049743044892,7.993199326301251) -- (2.6177719710537324,6.992455813948611) -- cycle;
\draw[scale=1] (0,2) node[left] {\footnotesize $r $};
\draw[scale=1] (0,2.5) node[left] {\footnotesize $h $};
\draw[scale=1] (3.1,9.0) node[above] {\footnotesize $L_2 $};
\draw[scale=1] (12.0,2.5) node[above] {\footnotesize $L_1 $};
\draw[line width=0.6pt,dashed,color=black] (0,2.6) -- (13.0,2.6);
\draw[line width=0.6pt,dashed,color=black] (0,2.0) -- (13.0,2.0);
\draw[line width=1.0pt,red,fill=yellow] (3.45,3.45) circle (3.0pt);
\draw[line width=1.0pt,blue,fill=yellow] (2.73,5.00) circle (3.0pt);
\draw[line width=1.0pt,blue,fill=yellow] (5.0,2.75) circle (3.0pt);
\draw[scale=1] (5,-0.3) node[below] {\footnotesize (i) $r=2, h=2.6$};
\end{tikzpicture}
\end{center}
\end{minipage}%
\caption{This figure shows the curves of $L_1$ and $L_2$ when $H^*\leq \bar y_h$ and $r<h.$ The three  intersection points are emphasized, and for the given choice of our parameters, they are   $(2.741,4.969), (3.424, 3.424)$ and $ (4.969,2.741).$
The shaded region represents the solution of $(x,y)\leq_{se} (F(x,y),F(y,x))$ in which $x\leq y.$     }\label{Fig-Intersections2}
\end{figure}

 \begin{remark}\label{feasible}\rm The inequalities $(x,y)\leq_{se} (F(x,y),F(y,x))$ and $x\leq y$ have a feasible solution in case (ii) and case (iii) of Lemma \ref{Lem-Intersections}, that is when $h>r$. This is essential for applying the embedding technique. Concerning the intersections between the curves $L_1$ and $L_2,$ let $g_1(t)=\frac{h_0}{1-f(t)},\; t>r.$ The function $g_1$ is decreasing in $t.$ Since an intersection point $(a,b)$ must satisfy $(a,b)=(g_1(b),g_1(a)),$ the fixed point of $g_1$ is the fixed point of $F,$ and the $2$-cycles of $g_1$ form the pseudo fixed points of $F.$ Since a period-doubling bifurcation occurs when the fixed point of $g_1$ loses its stability, the existence of pseudo-fixed points can be investigated through the local stability of the fixed point of $g_1.$
 \end{remark}

After establishing the machinery and lemmas needed, we can present one of our first main results. Recall that the fixed points of the embedded map $G$ are associated with the intersection points of $L_1$ and $L_2$. In particular, $G$ has a unique fixed point given by $\xi=\bar \xi_h$ when cases (i) and (ii) of Lemma \ref{Lem-Intersections} are satisfied, and $G$ has three fixed points in case (iii) of that Lemma. As done earlier, we denote the asymmetric fixed points of $G$ by $\bar \xi_1=(x^*,y^*,y^*,x^*)$ and $\bar \xi_2=(y^*,x^*,x^*,y^*)$ where  $x^*<y^*.$ In other words, $(x^*,y^*)$ and $(y^*,x^*)$ are the pseudo fixed points of $F.$

\begin{theorem}\label{Th-GlobalStabilit1}
Consider  $r_{2}$ as given in Eq. (\ref{Eq-r2}). Each of the following holds true for Eq. (\ref{Eq-Ricker3}):
\begin{description}
\item{(i)} If $r\leq r_2,$ then $\bar y_h$ is globally asymptotically stable.
\item{(ii)} If $r_2<r<h,$ then the compact box $[x^*,y^*]^2$ forms an absorbing region. In particular,
$$x^*\leq \liminf \{x_n\}\leq \limsup \{x_n\}\leq y^*.$$
\end{description}
\end{theorem}

\begin{proof}
(i) Since $r\leq r_{2},$ Lemma \ref{Lem-Intersections} guarantees a unique intersection between $L_1$ and $L_2,$ and consequently, the embedding $G$ has a unique fixed point, which we denoted by $\bar \xi_h = (\bar y_h, \bar y_h,\bar y_h, \bar y_h)$. Since $r_2<r_1$, $\bar y_h$ is LAS by Lemma \ref{stable-unstable}.  To show global stability, notice that the inequalities $(a,b)\leq_{se} (F(a,b),F(b,a))$ and $a<b$ have a feasible region (as illustrated in Part (ii) of Fig. \ref{Fig-Intersections1} and Remark \ref{feasible}). For any initial condition $(x_0,x_{-1})$ of Eq. (\ref{Eq-Ricker3}), there is $(a,b)$ in the feasible region such that Proposition \ref{Prop-GlobalStability} is applicable. This implies that for all $\xi_0=(y_0,y_{-1},y_0,y_{-1})$  such that
\begin{equation}\label{initial}\xi=(a,b,b,a)\leq_{se}\xi_0\leq_{se} (b,a,a,b)=\xi^t,
\end{equation}
$G^n(\xi_0)$ converges to $\bar \xi_h$ and Eq. (\ref{Eq-Ricker3}) with the initial condition $(y_0,y_{-1})$ converges to $\bar y_h.$ This verifies Part (i).
Part (ii) is similar and refers to Part (iii) of Lemma \ref{Lem-Intersections} to obtain the three fixed points of $G$ denoted by $\bar \xi_1,\bar \xi_2$ and $\bar \xi_h.$
We have the same scenario as in Condition \eqref{initial} for initial conditions $(x_0,x_{-1})$ and the existence of $(a,b)$ in the feasible region as above so that $G^n(\xi_0)$ must converge to a fixed point of $G,$ and by the embedding result, the orbit of Eq. (\ref{Eq-Ricker3}) that starts at $(y_0,y_{-1})$ will be eventually squeezed between $x^*$ and $y^*$ (Proposition \ref{Prop-GlobalStability}). This completes the proof.
\end{proof}
It is worth stressing that  $r_2<r_1$ always, and $0<r<r_1$ guarantees the local stability of $\bar y_h.$ However, when $r_2<r<\min\{h,r_1\}$, we obtain the pseudo fixed points, which cripple the embedding technique. Nevertheless, an absorbing region was obtained as given in Part (ii) of Theorem \ref{Th-GlobalStabilit1}. Finally, when $h<r<r_1,$ $\bar y_h$ is LAS. However, the absence of a feasible solution for the system $(x,y)\leq_{se} (F(x,y),F(y,x))$ and $x\leq y$ led to the unsuccessful outcome of our embedding approach for addressing global stability. Figure \ref{Fig-StabilityRegions1} illustrates and summarizes our stability results in the parameter space.

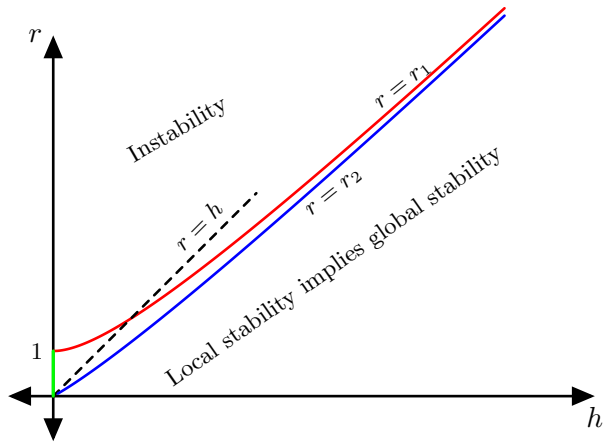
\begin{figure}[H]
\begin{center}
\begin{tikzpicture}[line cap=round,line join=round,>=triangle 45,x=1.0cm,y=1.0cm,scale=0.6]
\draw[-triangle 45, line width=1.0pt,scale=1] (0,0) -- (12.0,0) node[below] {$h$};
\draw[line width=1.0pt,-triangle 45] (0,0) -- (-1,0);
\draw[-triangle 45, line width=1.0pt,scale=1] (0,0) -- (0,8) node[left] {$r$};
\draw[line width=1.0pt,-triangle 45] (0,0) -- (0.0,-1);
\draw[line width=1pt,color=red,smooth,samples=100,domain=0:10.0] plot(\x,{\x+1-ln(\x+1)});
\draw[line width=1pt,color=black,dashed,domain=0:4.5] plot(\x,{\x});
\draw[line width=1pt,color=blue,smooth,samples=100,domain=0.1:10.0] plot(\x,{ln((\x*\x+4*\x)^0.5-\x)-ln(\x+(\x*\x+4*\x)^0.5)+\x/2+(\x*\x+4*\x)^0.5/2});
\draw[line width=1.2pt,color=green] (0,0)--(0,1);
\draw[scale=1] (6.5,2.5) node[above,rotate=35] {\footnotesize Local stability implies global stability};
\draw[scale=1] (8.0,6.6) node[above,rotate=39] {\footnotesize $r=r_1$};
\draw[scale=1] (6,4.8) node[below,rotate=39] {\footnotesize $r=r_2$};
\draw[scale=1] (0,1) node[left,rotate=0] {\footnotesize $1$};
\draw[scale=1] (3.5,3.5) node[above,rotate=45] {\footnotesize $r=h$};
\draw[scale=1] (3,5.5) node[above,rotate=30] {\footnotesize Instability};
\end{tikzpicture}
\end{center}
\caption{This figure shows the stability regions in the $(h,r)-$plane. The red curve is $r=r_1,$ which represents the solution of $\bar y_h=1+h,$ while the blue curve is $r=r_2,$ which represents the solution of $H^*=H^*f(H^*)+h.$ We have local stability in the region below $r=r_1$, but the embedding technique fails to address global stability when $r_2<r<h$. However, we conjecture that we also obtain global stability when $r_2<r<r_1$. Observe that global stability in the green part of the $r$-axis, i.e., $h=0$ and $0<r<1$ has been addressed in \cite{Ba-Ga-Kr2013}.}\label{Fig-StabilityRegions1}
\end{figure}

\section{Stability under periodic stocking}
In this section, we consider the Ricker model with delay and $p$-periodic stocking as given in Eq. (\ref{Eq-Ricker4}), i.e., the minimal period of the stocking sequence is $p$. Recall that we give ourselves the liberty to use $f(x)$ instead of $exp(r-x)$ for writing conveniences. Define the sequence of $2$-dimensional maps $\{T_j\},$
$$T_j\;:\; \mathbb{R}^2_+\to \mathbb{R}^2_+\quad \text{ as}\quad T_j(x,y)=(xf(y)+h_j,x).$$
The equation $X_{n+1}=T_{n\bmod p}(X_n)$ is just a vector form of Eq. (\ref{Eq-Ricker4}).
Since
\begin{align*}
y_{n+2}=&y_nf(y_n)f(y_{n-1})+h_0f(y_n)+h_1\\
\leq& (e^r+h_0)e^r+h_1,
\end{align*}
the composition operator $\widetilde{T}:=T_{p-1}\circ T_{p-2}\circ \cdots \circ T_0$ maps a nonempty convex compact set into itself. Therefore, $\widetilde{T}$ has a fixed point in that domain. Because $T_i(x,y)=T_j(x,y)$ if and only if $i=j,$ then a fixed point of  $\widetilde{T}$ must form a $p$-periodic solution of $X_{n+1}=T_{n\bmod p}(X_n),$ which reflects a $p$-periodic solution of Eq. (\ref{Eq-Ricker4}). We are interested in the global stability of the obtained $p$-periodic solution, and we focus on the case $p=2.$
The existing $2$-periodic solution is in fact the $2$-cycle of the system $[F_0,F_1]$, and we denote it by $\{\bar z_0=\bar z_0(h_0,h_1),\bar z_1=\bar z_1(h_0,h_1)\},$
where
\begin{equation}\label{Eq-2Cycle}
\bar z_0=\bar z_1f(\bar z_0)+h_1\quad \text{and}\quad \bar z_1=\bar z_0f(\bar z_1)+h_0.
\end{equation}
Observe that $\bar z_0>h_1$ and $\bar z_1>h_0.$ Furthermore,  $\bar z_0$ and $\bar z_1$ contribute to the formation of two fixed points of $\widetilde{T}=T_1\circ T_0,$  namely $\bar X_0:=(\bar z_0,\bar z_1)$ and $\bar X_1:=(\bar z_1,\bar z_0).$ Now, we give the following fact:

\begin{proposition}\label{Pr-h0>h1}
$h_0>h_1$ if and only if $\bar z_1>\bar z_0.$
\end{proposition}
\begin{proof}
From  Eqs. (\ref{Eq-2Cycle}), we obtain
$$\bar z_1-\bar z_0=\bar z_0(f(\bar z_1))-\bar z_1f(\bar z_0)+h_0-h_1.$$
This gives us
$$\frac{h_0-h_1}{\bar z_1-\bar z_0}=1-\bar z_0\frac{f(\bar z_1)-f(\bar z_0)}{\bar z_1-\bar z_0}+f(\bar z_0)>0.$$
Hence,   $sign(h_0-h_1)=sign(\bar z_1-\bar z_0),$ which completes the proof.
\end{proof}
\begin{remark}
Since $h_0<h_1$ and $h_1<h_0$ give similar behaviour, we focus on the case $h_1<h_0.$ In this case, we have $h_1<\bar z_0< \bar z_1.$
\end{remark}
 To investigate the stability of the $2$-cycle, it is enough to investigate the stability of the fixed point $\bar X_0=(\bar z_0,\bar z_1)$ under $\widetilde{T}=T_1\circ T_0.$ Since
$$\widetilde{T}(x,y)=\left[
           \begin{array}{c}
             (xf(y)+h_0)f(x)+h_1 \\
             xf(y)+h_0 \\
           \end{array}
         \right],$$
the Jacobian matrix at $\bar X_0$ is
$$J(\bar X_0)=\left[
                \begin{array}{cc}
                  \bar z_1f^\prime(\bar z_0)+f(\bar z_0)f(\bar z_1) & \bar z_0f(\bar z_0)f^\prime(\bar z_1) \\
                  f(\bar z_1) & \bar z_0f^\prime(\bar z_1) \\
                \end{array}
              \right].$$
The trace and determinant here are given by
$$Tr=\bar z_1f^\prime(\bar z_0)+\bar z_0f^\prime(\bar z_1)+f(\bar z_0)f(\bar z_1)\quad \text{and}\quad Det=\bar z_0\bar z_1 f^\prime(\bar z_1)f^\prime(\bar z_0).$$
We re-write the two expressions as
\begin{align}
Tr=&(h_1-\bar z_0)+(h_0-\bar z_1)+\left(1-\frac{h_0}{\bar z_1}\right)\left(1-\frac{h_1}{\bar z_0}\right), \label{Eq-T2}\\
Det=& (\bar z_1-h_0)(\bar z_0-h_1).  \label{Eq-D2}
\end{align}
Obviously, $Det>0$ and, $Det<1$ when $(\bar z_1-h_0)(\bar z_0-h_1)<1.$ It turns out that this condition is sufficient for the local stability of the $2$-cycle as the following result shows.

\begin{lemma} \label{Lem-LAS-2Cycle}
Let $h_0\neq h_1$ and both are non-negative. The $2$-cycle of $[F_0,F_1]$ (i.e., $\{\bar z_0,\bar z_1\}$) is LAS stable if
$(\bar z_1-h_0)(\bar z_0-h_1)<1.$
\end{lemma}
\begin{proof}
We check the Jury conditions here. We have $0<Det<1.$ It remains to show that
$$Det-Tr>-1\quad \text{and}\quad Det+Tr>-1.$$
The first inequality is satisfied because
\begin{equation}\label{Eq-D-T+1}
Det-Tr+1=(\bar z_1-h_0+1)(\bar z_0-h_1+1)-\left(1-\frac{h_0}{\bar z_1}\right)\left(1-\frac{h_1}{\bar z_0}\right)>0.
\end{equation}
Now, we use this fact to write
$$(\bar z_1-h_0+1)(\bar z_0-h_1+1)>\left(1-\frac{h_0}{\bar z_1}\right)\left(1-\frac{h_1}{\bar z_0}\right),$$
or equivalently,
$$\bar z_1\left(1-\frac{1}{\bar z_1-h_0}\right)\bar z_0\left(1-\frac{1}{\bar z_0-h_1}\right)>1.$$
Next, write
\begin{align*}
Det+Tr+1=&(\bar z_0-h_1-1)(\bar z_1-h_0-1)+\left(1-\frac{h_1}{z_0}\right)\left(1-\frac{h_0}{z_1}\right)\\
=&\left(1-\frac{h_1}{z_0}\right)\left(1-\frac{h_0}{z_1}\right)\left[\bar z_1\left(1-\frac{1}{\bar z_1-h_0}\right)\bar z_0\left(1-\frac{1}{\bar z_0-h_1}\right)+1\right],
\end{align*}
which is positive. This completes the proof.
\end{proof}
Note that the condition $(\bar z_1-h_0)(\bar z_0-h_1)<1$ can be replaced by the more simple (but unnecessary) condition $r\leq 1.$ Also, the simple conditions $\bar z_j \leq h_{j+1}+1$ for $j=0,1$ are sufficient to make $Det<1.$  For the reader's convenience, we summarize some special cases in the following corollary:
\begin{corollary}
Each of the following cases is valid:
\begin{description}
\item{(i)} If $r\leq 1,$ then $Det<1.$
\item{(ii)} If $\bar z_j \leq h_{j+1}+1$ for all $j=0,1,$ then the $2$-cycle is LAS.
\item{(iii)} If $\bar z_j > h_{j+1}+1$ for all $j=0,1,$ then the $2$-cycle is unstable.
\end{description}
\end{corollary}
Now, we provide some examples that illustrate the validity of Lemma \ref{Lem-LAS-2Cycle}.
\begin{example}\label{Ex-PeriodDoubling}
Consider Eq. (\ref{Eq-Ricker4}) in each one of the following cases:
\begin{description}
\item{(i)} Let $h_0=2.0, h_1=6.444$ and $r=3.0.$ We obtain $\bar z_0\approx 6.635$ and $\bar z_1\approx 7.237.$ In this case, the condition of Lemma \ref{Lem-LAS-2Cycle} is not satisfied. Furthermore, $Det\approx 1.000$ and $Tr\approx -5.407.$ This makes $Det+Tr+1<0,$ and consquently, the $2$-cycle $\{\bar z_0,\bar z_1\}$ of the system $[F_0,F_1]$ is unstable. In fact, we obtain the $4$-cycle
$$\{7.049,\; 19.611,\; 6.479,\; 2.000\}.$$
\item{(ii)} Let $r=2.0,h_0=2.156$ and $h_1=2.720.$ We obtain $\bar z_0\approx 3.462$ and $\bar z_1\approx 3.199.$ In this case, Lemma \ref{Lem-LAS-2Cycle} is satisfied, and the $2$-cycle $\{\bar z_0,\bar z_1\}$ of the system $[F_0,F_1]$ is LAS. Furthermore, $Det\approx 0.774,$ $Tr\approx -1.715$ and $Det+Tr+1\approx 0.059>0,$ and consquenlty, the eigenvalues are $\approx -0.858 \pm 0.196i,$ which are within the unit disk.
\item{(iii)} Let $h_0=0.820$ and $h_1=1.800.$ We let $r=1.5$ to obtain $\bar z_0\approx 2.552$ and $\bar z_1\approx 2.151.$ In this case, the condition of Lemma \ref{Lem-LAS-2Cycle} is not satisfied. However, by considering $r$ as a bifurcation parameter, a Neimark-Sacker bifurcation occurs near $r=1.5.$ Indeed, the $2$-cycle bifurcates into two invariant curves that serve as one attractor.  Fig. \ref{Fig-NeimarkSacker} illustrates this scenario.
\end{description}
\end{example}
Based on Lemma \ref{Lem-LAS-2Cycle}, The $2$-cycle of Eq. (\ref{Eq-Ricker4}) has the potential to lose its stability after going through a Neimark-Sacker bifurcation. The period-doubling observed in Part (i) of Example \ref{Ex-PeriodDoubling} comes after the $2$-cycle exits the stability region through the Neimark-Sacker bifurcations. Figure \ref{Fig-NeimarkSacker} Shows an illustrative case of the Neimark-Sacker bifurcation that occurs.

\begin{figure}[htpb]
    \centering
    \includegraphics[width=0.6\textwidth]{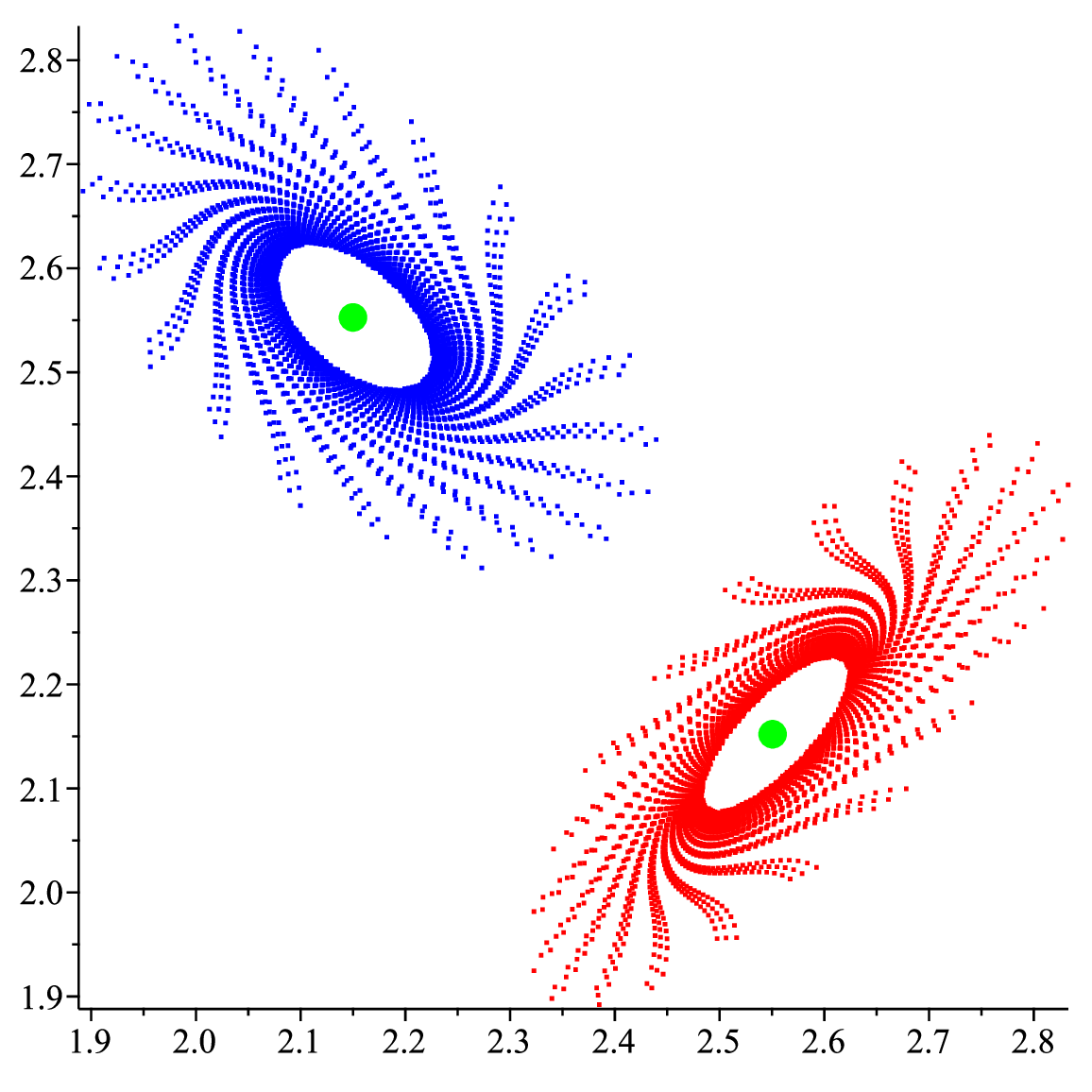}
    \caption{Let $h_0=0.820$ and $h_1=1.800.$ When $r=1.5,$ the $2$-cycle bifurcates into two curves through a Neimark-Sacker bifurcation. The two curves serve as one attractor. The red represents the even terms in the orbit, while the blue represents the odd terms. The green solid circles represent the $2$-cycle. The horizontal axis is for $x_{n}$, and the vertical axis is for $x_{n-1}.$  }
    \label{Fig-NeimarkSacker}
\end{figure}

Next, we proceed to find conditions under which we obtain global stability.  Recall $F_j(x,y)=xf(y)+h_j$ for $j=0,1.$ First, we investigate the existence of pseudo $2$-cycles. Solving the equation $G_1(G_0(\xi))=\xi,$ where $\xi=(x,y,u,v)$ gives us (see Subsection 2.3)
$$(x,y)=(F_1(v,u),F_0(u,v))\quad \text{and}\quad (u,v)=(F_1(y,x),F_0(x,y)).$$
Since each value of $(x,y)$ determines a unique value of $(u,v),$ it is enough to focus on the equations
\begin{equation} \label{Eq-artificialCycles}
\begin{cases}
x-h_1=& (xf(y)+h_0)f(yf(x)+h_1)\\
y-h_0=&(yf(x)+h_1)f(xf(y)+h_0).
\end{cases}
\end{equation}
Observe that the $2$-cycle $\{\bar z_0,\bar z_1\}$  (see Eq. (\ref{Eq-2Cycle})) satisfies System (\ref{Eq-artificialCycles}). For other solutions, our best bargain here is to investigate the various scenarios based on the curves of the two equations. Let $\ell_1$ and $\ell_2$ be the curves of the first and second equations in System \ref{Eq-artificialCycles}, respectively. The $\ell_1$-curve has a horizontal asymptote at $y=2r-h_1$ and a vertical asymptote at $x=h_1.$ Similarly, the $\ell_2$-curve has a horizontal asymptote at $y=h_0$ and a vertical asymptote at $x=2r-h_0.$ Multiple intersections between $\ell_1$ and $\ell_2$ guarantee the existence of artificial cycles, and in this case, the embedding technique fails to help us establish global stability. Thus, a unique intersection between $\ell_1$ and $\ell_2$ is crucial to our embedding strategy. Next, we appeal to the results of Subsection \ref{SubSection-2.3} and Proposition \ref{Pr-GSConditions} in particular. Our primary objective is to identify values for $(a,b)$ such that  $a<b$ and the condition \eqref{In-ConditionsOnAandB} of Proposition \ref{Pr-GSConditions} are satisfied.
By Part (ii) and Part (iii) of Lemma \ref{Lem-Intersections}, and the consequent Remark \ref{feasible},  the inequalities $(a,b)<_{se} (F_0(a,b),F_0(b,a))$ have a feasible region when $h_0>r.$ However, we need the feasible region to overlap with the feasible region of the inequalities
\begin{equation}\label{In-FeasibleRegionPeriodic2}
a\leq  F_1(F_0(a,b),b)\quad\text{and}\quad  b\geq F_1(F_0(b,a),a).
\end{equation}
 A quick observation here is that if $(a,b)<_{se} (F_j(a,b),F_j(b,a))$ for each $j,$ then the inequalities in \eqref{In-FeasibleRegionPeriodic2} have a feasible region; however, we can avoid this strong constraint.  Before we proceed, we give some illustrative computational examples.

\begin{example}\label{Ex-LastExample}
\begin{description}
\item{(i)}
Consider $r=1.0,h_0=2.0$ and $h_1=1.5.$ The $2$-cycle is $\{\bar z_0,\bar z_1\}\approx \{2.230,2.498\}$. From Eqs. (\ref{Eq-T2}) and (\ref{Eq-D2}), we obtain $Tr\approx -1.163,$ and $Det\approx 0.364.$ The eigenvalues are $\approx -0.582 \pm 0.161i,$ which are within the unit disk. Thus, the $2$-cycle $\{\bar z_0,\bar z_1\} $ of the system $[F_0,F_1]$ is locally asymptotically stable. The curves $\ell_1$ and $\ell_2$ are given in Part (i) of Fig. \ref{Fig-Intersections3}.
\item{(ii)} When $r=1.0,h_0=2.0$ and $h_1=1.0,$ we obtain the $2$-cycle $[\bar z_0,\bar z_1]\approx [2.455,1.950]$ of the system $[F_0,F_1].$ The trace and determinant from Eqs. (\ref{Eq-T2}) and
 (\ref{Eq-D2}) are given by $Tr\approx -1.42$ and $Det\approx -0.0732.$
The eigenvalues are $\approx -1.470$ and   $0.050.$ Obviously, the $2$-cycle of the system $[F_0,F_1]$ is unstable. Next, we investigate the solution of the system $G_1(G_0(\xi))=\xi.$ As clarified in Subsection \ref{SubSection-2.3}, we explore the following scenarios:
\begin{itemize}
\item $G_1(\xi)=G_0(\xi)=\xi,$ where $\xi=(x,y,y,x).$ $x=y$ leads to an equilibrium solution of Eq. (\ref{Eq-Ricker4}), while $x\neq y$ leads to an artificial equilibrium solution. Both scenarios are not possible since $h_0\neq h_1.$

\item $\xi_0=(x,x,y,y), x\neq y, G_0(\xi_0)=\xi_1=(y,y,x,x)$ and $G_1(\xi_1)=\xi_0$. This means $\{x,y\}$ is a $2$-cycle for each one of the  one-dimensional maps $f_j(x)=F_j(x,x),j=0,1.$ Again, this is not possible since $h_0\neq h_1.$

\item $\xi_0=(x,y,x,y), x\neq y, G_0(\xi_0)=\xi_1=(y,x,y,x)$ and $G_1(\xi_1)=\xi_0.$ This means $\{\xi_0,\xi_1\}$ is a $2$-cycle of $[G_0,G_1]$ and $\{x,y\}$ is a $2$-cycle of $[F_0,F_1].$ This is always possible in our system since we verified the existence of a $p$-cycle for the $p$-periodic system $[F_0,F_1,\ldots, F_{p-1}]$. For the specific choice of our parameters, we already computed the $2$-cycle above.

\item  $\bar \xi_0=(\bar x,\bar y,\bar u,\bar v), \bar x\neq \bar y, (\bar x,\bar y)\neq (\bar u,\bar v), G_0(\bar \xi_0)=\bar \xi_1=(\bar v,\bar u,\bar y,\bar x)$ and $G_1(\bar \xi_1)=\bar \xi_0.$  This scenario is possible under certain choices of the parameters, and it leads to what we call artificial cycles. Indeed, when $r=1.0,h_0=2.0,$  and $h_1=1.0,$ we obtain the following: The fixed points of $G_{10}=G_1\circ G_0$ are $\bar \xi_0=(\bar x,\bar y,\bar u,\bar v)$ and $\bar \xi_1=(\bar u,\bar v,\bar x,\bar y),$ where
    $$\bar x\approx 1.109,\; \bar y\approx 3.306,\; \bar u\approx 3.966 \quad \text{and}\quad \bar v\approx 2.110.$$
    On the other hand, the fixed points of $G_{01}=G_0\circ G_1$ are $\bar \eta_0=(\bar v,\bar u,\bar y,\bar x)$ and $\bar \eta_1=(\bar y,\bar x,\bar v,\bar u).$ Therefore, $[G_0,G_1]$ has two $2$-cycles that are not generated by $2$-cycles of $[F_0,F_1],$ namely $\{\bar \xi_0,\bar \eta_0\}$ and $\{\bar \xi_1,\bar \eta_1\}.$ Those cycles lead to what we called artificial cycles of $[F_0,F_1],$ namely $\{(\bar x,\bar y),(\bar v,\bar u)\}$ and $\{(\bar u,\bar v),(\bar y,\bar x)\}.$
  In this case,  the embedding technique fails to tackle global stability in the system $[F_0,F_1]$. However, we can eventually squeeze the orbits of $G_{10}$ between $\bar \xi_0$ and $\bar \xi_1$ (with respect to $\leq_{se}$). Similarly, the orbits of $G_{01}$ can be eventually squeezed between $\bar \eta_0$ and $\bar \eta_1$ (again with respect to $\leq_{se}$). When it comes to the system $[F_0,F_1],$ the even terms of the orbits are eventually squeezed between $\bar x$ and $\bar u$, while the odd terms of the orbits are eventually squeezed between $\bar v$ and $\bar y.$
\end{itemize}
\end{description}
\end{example}

\definecolor{ffqqqq}{rgb}{1.,0.,0.}
\definecolor{qqqqff}{rgb}{0,0,1}
\definecolor{zzttqq}{rgb}{0.6,0.2,0}
\begin{figure}[h!]
\centering
\begin{minipage}[t]{0.5\textwidth}
\raggedright
\begin{center}
\begin{tikzpicture}[line cap=round,line join=round,>=triangle 45,x=1.0cm,y=1.0cm,scale=0.5]
\draw[-triangle 45, line width=1.0pt,scale=1] (0,0) -- (10.0,0) node[below] {$x$};
\draw[line width=1.0pt,-triangle 45] (0,0) -- (-1,0);
\draw[-triangle 45, line width=1.0pt,scale=1] (0,0) -- (0,10) node[left] {$y$};
\draw[line width=1.0pt,-triangle 45] (0,0) -- (0.0,-1);
\draw[line width=1.2pt,domain=0:9.0,smooth,variable=\x,green] plot ({\x}, {\x)});
\fill[line width=2pt,color=zzttqq,fill=zzttqq,fill opacity=0.5] (1.1,9.1) -- (1.15,8.024) -- (1.2,7.435) -- (1.25,6.9389) -- (1.3,5.985)--(1.35,5) -- (1.4,4.55) -- (1.45,4.0) --(1.5,4)-- (1.48,5.5166) -- (1.47,6.5) -- (1.4,7.81) -- (1.38,8.66) -- (1.36,9.1) -- cycle;
\draw [color=red,line width=1pt] plot [smooth] coordinates {(1.5,9) (1.51,7.99) (1.6,4.587) (1.8,3.30) (1.9,3.1) (2.0,2.8) (2.2,2.53) (2.5,2.24) (2.7,2.1) (3.0,1.9) (3.5,1.67) (4.0,1.49) (4.5,1.36) (5,1.24) (6,1.1) (7,0.98) (8,0.91) (9,0.88)};
\draw [color=blue,line width=1pt] plot [smooth] coordinates {(0.3,9) (0.383,8.0) (0.4,7.6) (0.5,6.3) (0.56,5.9) (0.61,5.5) (0.7,5.0) (0.8,4.5) (0.95,4.0) (1.16,3.5) (1.495,3.0) (2.23,2.5) (2.96,2.3) (3.7,2.2) (5.2,2.1) (7,2.05) (8,2.0) (9,2.0)};
\draw[scale=1] (0.7,8.5) node[above] {\footnotesize $\ell_2 $};
\draw[scale=1] (2.0,8.7) node[above] {\footnotesize $\ell_1 $};
\draw[scale=1] (0,1.95) node[left] {\footnotesize $h_0 $};
\draw[scale=1] (0,0.5) node[left] {\footnotesize $2r-h_1 $};
\draw[line width=0.6pt,dashed,color=black] (0,0.7) -- (9.0,0.7);
\draw[line width=0.6pt,dashed,color=black] (0,2.0) -- (9.0,2.0);
\draw[line width=1.0pt,red,fill=yellow] (2.23,2.498) circle (2.5pt);
\draw[scale=1] (5,-0.7) node[below] {\footnotesize (i) $r=1.0, h_0=2.0$ and $h_1=1.5$};
\end{tikzpicture}
\end{center}
\end{minipage}%
\begin{minipage}[t]{0.5\textwidth}
\raggedright
\begin{center}
\begin{tikzpicture}[line cap=round,line join=round,>=triangle 45,x=1.0cm,y=1.0cm,scale=0.5]
\draw[-triangle 45, line width=1.0pt,scale=1] (0,0) -- (10.0,0) node[below] {$x$};
\draw[line width=1.0pt,-triangle 45] (0,0) -- (-1,0);
\draw[-triangle 45, line width=1.0pt,scale=1] (0,0) -- (0,10) node[left] {$y$};
\draw[line width=1.0pt,-triangle 45] (0,0) -- (0.0,-1);
\draw[line width=1.2pt,domain=0:9.0,smooth,variable=\x,green] plot ({\x}, {\x)});
\draw [color=blue,line width=1pt] plot [smooth] coordinates {(0.3056,8.98947) (0.3522,7.96957) (0.411,7.0079) (0.497,6.012) (0.636,4.9526) (0.741,4.410) (0.831,4.04768) (0.951,3.674) (1.15998,3.20998)
(1.1599840471258906,3.2099793976512663) (1.3791304718838568,2.8892190948029763) (1.3791304718838568,2.8892190948029763) (1.671338709363464,2.6171443335007334) (2.0010044568692624,2.4320033389501656) (2.3831995808772213,2.303937560304738) (2.7760996212951765,2.2238154144734823) (3.180090582313232,2.1709020886369235) (3.5803877049466935,2.1352525409215493) (4.321840496662944,2.0929187845360855) (5.002121622623495,2.068475451542712) (6.0008151055953896,2.0452158932656572) (7.000368076532915,2.0304049549940246) (8.000170916398805,2.020626615462483) (9.000073924511108,2.0140656110905697) (10.000035260475121,2.009625702286068)};
\draw [color=red,line width=1pt] plot [smooth] coordinates {(1.000,9.000) (1.000,8.000) (1.0018,7.000) (1.0051,6.000) (1.0146,5.000) (1.0326940851950859,4.270495283700743) (1.0446574616948663,4.002296070964245) (1.0972085776143057,3.386981145737016) (1.1877312432948064,2.952130901598307) (1.383897250779989,2.612664258846495) (1.7199621019707987,2.4697080183734545) (2.060047574872888,2.4545652984651367) (2.4211706001954827,2.4502475889625304) (2.7775730452444725,2.4158399945096334) (3.123,2.3474) (3.4859,2.251) (3.815,2.154) (4.3066,2.01544) (5.015135265619371,1.8464639958687687) (6.017574952187674,1.671892630946028) (6.9786149801073964,1.5560014982018753) (7.996789281377806,1.4683641260959845) (8.957696473551817,1.4071714833970774) (9.999402142750256,1.3564852069214748)};
\draw[scale=1] (2.0,8.5) node[above] {\footnotesize $\ell_1 $};
\draw[scale=1] (9.0,2.0) node[above] {\footnotesize $\ell_2 $};
\draw[scale=1] (0,1) node[left] {\footnotesize $2r-h_1 $};
\draw[scale=1] (0,2) node[left] {\footnotesize $h_0 $};
\draw[line width=0.6pt,dashed,color=black] (0,1.95) -- (9.0,1.95);
\draw[line width=0.6pt,dashed,color=black] (0,1.0) -- (9.5,1.0);
\draw[line width=1.0pt,red,fill=yellow] (1.95,2.455) circle (2.5pt);
\draw[scale=1] (5,-0.7) node[below] {\footnotesize (ii) $r=1.0, h_0=2.0$ and $h_1=1.0$};
\end{tikzpicture}
\end{center}
\end{minipage}%
\caption{Part (i) of this figure shows the curves $\ell_1$ and $\ell_2$ of the first and second equations in System (\ref{Eq-artificialCycles}), respectively. The unique intersection is the $2$-cycle $\{\bar z_0,\bar z_1\}$ of the system $[F_0,F_1],$ where $\bar z_0\approx 2.498$  and $\bar z_1\approx 2.230.$ The shaded thin region represents the points $(a,b)$ that satisfy $(a,b,b,a)<_{se}G_1(G_0(a,b,b,a))$ and $a<b.$  Part (ii) shows the same curves but when $h_1$ is changed to $1.0.$ The three intersections between the curves represent the $2$-cycle $\{\bar z_0\approx 2.455,\bar z_1\approx 1.950\}$ of the system $[F_0,F_1]$ and two other artificial cycles generated from the intersections $(1.109,3.306)$ and $(3.966,2.110)$.  }\label{Fig-Intersections3}
\end{figure}
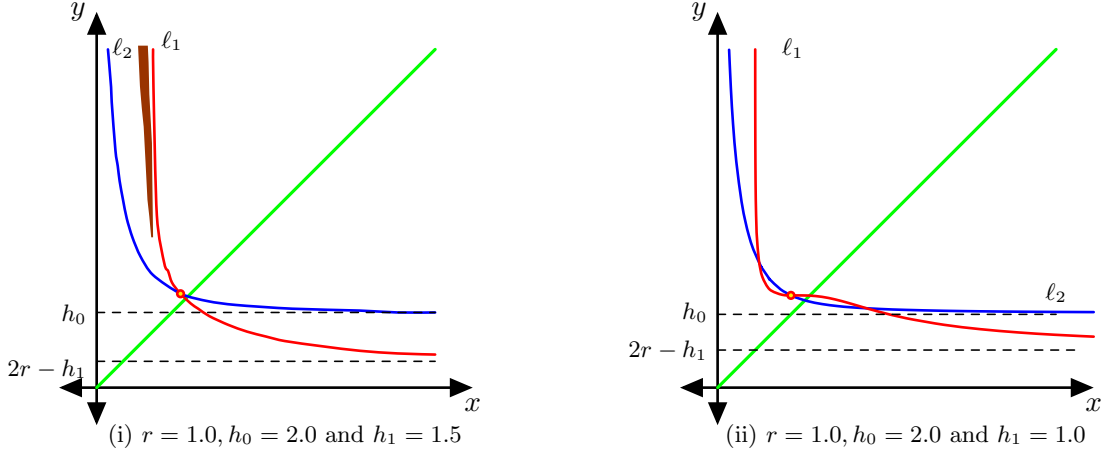
 Now, we revisit the system of inequalities in (\ref{In-FeasibleRegionPeriodic2}), and re-write them as
 $$\begin{cases}
 a(1-(f(b)))^2\leq& h_0f(b)+h_1\\
 b(1-(f(a)))^2\leq& h_0f(a)+h_1.
 \end{cases}$$
 For $a,b>r,$ the factors on the left-hand side are positive, which motivate us to define the map $g_2(t)=\frac{h_0f(t)+h_1}{1-(f(t))^2 },\; t>r.$  Also, recall from Remark \ref{feasible} that we defined the map $g_1(t)=\frac{h_0}{1-f(t)}.$ Both $g_1$ and $g_2$ are decreasing in $t.$ The inequalities in (\ref{In-FeasibleRegionPeriodic2}) become $a\leq g_2(b)$ and $b\geq g_2(a).$ It is straightforward to check the following result.
 \begin{lemma}\label{Lem-ConditionsOn(a,b)}
 Assume $h_0>r.$ There exists an unbounded set of points $(a,b),\; a<b$ that satisfy
 $$a\leq g_1(b)<g_2(b)<g_2(a)\quad \text{whenever}\quad h_0<h_1$$
 and
 $$ b\geq g_1(a)>g_2(a)>g_2(b)\quad \text{whenever}\quad h_0>h_1.$$
 \end{lemma}

 Finally, we arrive at the main result of this section, which provides a generalization of Part (i) of Theorem \ref{Th-GlobalStabilit1}.

\begin{theorem}\label{Th-MainResult}
Consider Eq. (\ref{Eq-Ricker4}) with $h_0\neq h_1.$ If both $h_0,h_1>r$ and the solution of System (\ref{Eq-artificialCycles}) is unique, then the $2$-cycle $\{\bar z_0,\bar z_1\}$ is global asymptotically stable.
\end{theorem}

\begin{proof}
For each initial condition $(x_0,x_{-1})$ of Eq. \eqref{Eq-Ricker4}, we need to find a point $(a,b)$ that makes Proposition \ref{Pr-GSConditions} applicable. Without loss of generality, we can consider the initial conditions to be $(x_2,x_1).$ In this case, we guarantee that $x_1>h_0>r$ and $x_2>h_1>r.$  If $h_0>h_1,$ we consider $(a,b)$ such that $a<b, b=g_1(a)$ and $a$ is sufficiently small so that $r<a<h_1.$ In this case and based on Lemma  \ref{Lem-ConditionsOn(a,b)}, we obtain
$$b\geq g_1(a)>g_2(a)>g_2(b)>h_1>a.$$
If $h_0<h_1,$ we consider $(a,b)$ such that $a<b, a=g_1(b)$ and $b$ is sufficiently large so that $g_2(h_0)<b.$ In this case and based on Lemma  \ref{Lem-ConditionsOn(a,b)}, we obtain
$$a\leq g_1(b)<g_2(b)<g_2(a)=g_2(g_1(b))<g_2(h_0)<b.$$
Observe that this choice of $(a,b)$ can be done to have $(x_2,x_1)\in [a,b]^2.$ All the hypotheses of Proposition \ref{Pr-GSConditions} are satisfied, and we invoke the Lemma to obtain the result.
\end{proof}
The uniqueness condition in Theorem \ref{Th-MainResult} is to guarantee that the global attractor is the $2$-cycle of $[F_0,F_1].$ However, if we require $h_0,h_1>r,$ then the theorem gives us either a globally attracting $2$-cycle or a trapping box in which its boundaries are determined by two artificial $2$-cycles of $[F_0,F_1].$

\section{Conclusion}
 We investigated two systems in this paper: the Ricker model with a time delay and constant stocking $x_{n+1}=F(x_n,x_{n-1})=x_n\exp(r-x_{n-1})+h$, and the system involving periodic stocking $x_{n+1}=F_n(x_n,x_{n-1})=x_n\exp(r-x_{n-1})+h_n.$  In both cases, we embedded the $2$-dimensional system into a $4$-dimensional monotonic system with respect to a partially ordered set. A unique positive equilibrium exists in the first system, which we denote by $\bar y_h.$ The equilibrium $\bar y_h$ transitions into a cycle in the second system, and the cycle inherits the system's period. We utilized the embedding technique to analyze global stability in both systems.
  \\

  In the constant stocking case, we obtained local stability when $\bar y_h<1+h$, which leads to the condition $0<r<r_1=h+1-\ln(h+1).$ On the other hand, we proved global stability (see Theorem \ref{Th-GlobalStabilit1}) under the constraint
$$\bar y_h<\frac{1}{2}\left(h+\sqrt{h^2+4h}\right),$$
which leads to the condition
$$0<r<r_2=\frac{1}{2}\left(h+\sqrt{h^2+4h}\right)+\ln\left(1-\frac{2h}{h+\sqrt{h^2+4h}}\right).$$
It's worth noting that the global stability condition we established is more stringent than the local stability condition, i.e., $r_2<r_1.$ Therefore, the common assertion that ``local stability implies global stability" is still an open problem for all values of $r$ between $r_1$ and $r_2$ as we illustrated in Fig. \ref{Fig-StabilityRegions1}.
\\

For the $2$-periodic system $x_{n+1}=F_{n\bmod 2}(x_n,x_{n-1}),$ we established conditions that ensure both local and global stability of the existing $2$-cycle (see Theorem \ref{Th-MainResult}). We illustrated a scenario in which the $2$-cycle undergoes a Neimark-Sacker bifurcation, leading to the emergence of two invariant curves that act as a unified attractor (see Example \ref{Ex-PeriodDoubling}). Finding the explicit form of the cycle is not possible, but we provided illustrative numerical examples that clarify various possible scenarios (see Example \ref{Ex-LastExample}).
As in the constant stocking case, sufficient conditions for the global stability of the $2$-cycle are obtained in the $2$-periodic case. Nonetheless, identifying necessary and sufficient conditions remains an open problem. We also conjecture here that the local stability of the $2$-cycle is sufficient to ensure its global stability.
\\

\noindent{\textbf{Acknowledgment:}} The authors express their gratitude to the anonymous reviewers for their valuable comments. This work does not have any conflicts of interest. There are no funders to report for this submission. 

\end{document}